\documentclass[11pt,a4paper]{article}

\usepackage{epsf,epsfig,amsfonts,amsgen,amsmath,amstext,amsbsy,amsopn,amsthm,cases,listings,color
}
\usepackage{ebezier,eepic}
\usepackage{color}
\usepackage{multirow}
\usepackage{epstopdf}
\usepackage{graphicx}
\usepackage{pgf,tikz}
\usepackage{mathrsfs}
\usepackage[marginal]{footmisc}
\usepackage{enumitem}
\usepackage[titletoc]{appendix}
\usepackage{booktabs}
\usepackage{url}
\usepackage{subfigure}
\usepackage{mathrsfs}
\usetikzlibrary{arrows}

\allowdisplaybreaks[1]

\definecolor{uuuuuu}{rgb}{0.27,0.27,0.27}
\definecolor{sqsqsq}{rgb}{0.1255,0.1255,0.1255}

\newtheorem{definition}{Definition} [section]
\newtheorem{theorem}[definition]{Theorem}
\newtheorem{lemma}[definition]{Lemma}
\newtheorem{proposition}[definition]{Proposition}

\newtheorem{corollary}[definition]{Corollary}
\newtheorem{conjecture}[definition]{Conjecture}
\newtheorem{claim}[definition]{Claim}
\newtheorem{problem}[definition]{Problem}

\newtheorem{fact}[definition]{Fact}
\newenvironment{pf}{\noindent{\bf Proof}}

\begin{document}
\title{\bf\Large Independent sets in hypergraphs omitting an intersection}

\date{\today}

\author{
Tom Bohman
\thanks{
Department of Mathematical Sciences,
Carnegie Mellon University,
Pittsburgh, PA, 15213
USA.
email: tbohman@math.cmu.edu.}
\and
Xizhi Liu
\thanks{Department of Mathematics, Statistics, and Computer Science, University of Illinois, Chicago, IL, 60607 USA.
email: xliu246@uic.edu.
Research partially supported by NSF awards DMS-1763317 and DMS-1952767.
}
\and
Dhruv Mubayi
\thanks{Department of Mathematics, Statistics, and Computer Science, University of Illinois, Chicago, IL, 60607 USA.
email: mubayi@uic.edu.
Research partially supported by NSF awards DMS-1763317 and DMS-1952767.}
}
\maketitle
\begin{abstract}
A $k$-uniform hypergraph with $n$ vertices is an $(n,k,\ell)$-omitting system if it does not contain
two edges whose intersection has size exactly $\ell$.  If in addition it does  not contain
two edges whose intersection has size greater than  $\ell$, then it is an $(n,k,\ell)$-system.
R\"{o}dl and \v{S}i\v{n}ajov\'{a} proved  a lower bound for the independence number of $(n,k,\ell)$-systems that is sharp in order of magnitude
for fixed $2 \le \ell \le k-1$. We consider the same question for the larger class of $(n,k,\ell)$-omitting systems.

For $k\le 2\ell+1$, we believe that the behavior is similar to the case of $(n,k,\ell)$-systems and prove a nontrivial lower bound for the first open case $\ell=k-2$.  For $k>2\ell+1$ we give new lower and upper bounds which show that the minimum independence number of $(n,k,\ell)$-omitting systems has a very different behavior than for $(n,k,\ell)$-systems. Our lower bound for $\ell=k-2$ uses some adaptations of the random greedy independent set algorithm, and our upper bounds (constructions) for $k> 2\ell+1$ are obtained from some pseudorandom graphs.

We also prove some related results where we forbid more than two edges with a prescribed common intersection size and this leads to some applications in Ramsey theory.
For example, we obtain good bounds for the Ramsey number $r_{k}(F^{k},t)$,
where $F^{k}$ is the $k$-uniform Fan.
Here the behavior is quite different than the case $k=2$ which reduces to the classical graph Ramsey number $r(3,t)$.
\end{abstract}

\section{Introduction}
For a finite set $V$ and $k \ge 2$ denote by $\binom{V}{k}$ the collection of all $k$-subsets of $V$.
A $k$-uniform hypergraph ($k$-graph) $\mathcal{H}$ is a family of $k$-subsets of finite set
which is called the vertex set of $\mathcal{H}$ and is denoted by $V(\mathcal{H})$.
A set $I \subset V(\mathcal{H})$ is {\em independent} in $\mathcal{H}$ if it contains no edge of $\mathcal{H}$.
The {\em independence number} of $\mathcal{H}$, denoted by $\alpha(\mathcal{H})$, is the maximum size of an independent set in $\mathcal{H}$.
For every $v \in V(\mathcal{H})$ the {\em degree} $d_{\mathcal{H}}(v)$ of $v$
in $\mathcal{H}$ is the number of edges in $\mathcal{H}$ that contain $v$.
Denote by $d(\mathcal{H})$ and $\Delta(\mathcal{H})$ the average degree and the maximum degree of $\mathcal{H}$, respectively.

An old result of Tur\'{a}n \cite{TU41} implies that $\alpha(G) \ge n/(d+1)$ for every graph $G$ on $n$ vertices with average degree $d$.
Later, Spencer \cite{SP72} extended Tur\'{a}n's result and proved that for all $k \ge 3$ every $n$-vertex $k$-graph
$\mathcal{H}$ with  average degree $d$ satisfies 
\begin{align}\label{equ:Spencer-bound}
\alpha(\mathcal{H}) \ge c_k \frac{n}{d^{1/(k-1)}}
\end{align}
for some constant $c_k > 0$.

The bound for $\alpha(\mathcal{H})$ can be improved if we forbid some family $\mathcal{F}$ of hypergraphs in $\mathcal{H}$.
For $\ell \ge 2$ a (Berge) {\em cycle} of length $\ell$ in $\mathcal{H}$ is a collection of $\ell$ edges $E_1,\ldots,E_{\ell} \in \mathcal{H}$
such that there exists $\ell$ distinct vertices $v_1,\ldots,v_{\ell}$ with $v_i \in E_{i}\cap E_{i+1}$ for $i\in[\ell-1]$
and $v_{\ell} \in E_{\ell} \cap E_{1}$.
A seminal result of Ajtai, Koml\'{o}s, Pintz, Spencer, and Szemer\'{e}di \cite{AKPSS82}
states that for every $n$-vertex $k$-graph $\mathcal{H}$ with average degree $d$ that
contains no cycles of length $2,3,$ and $4$, there exists a constant $c'_{k}>0$ such that
\begin{align}\label{equ:AKPSS-bound}
\alpha(\mathcal{H}) \ge c'_{k} \frac{n}{d^{1/(k-1)}}(\log d)^{1/(k-1)}.
\end{align}
Moreover, this is tight apart from $c_k'$.

Spencer \cite{NR90} conjectured and Duke, Lefmann, and R\"{o}dl \cite{DLR95}
proved that the same conclusion holds even if $\mathcal{H}$ just contains no cycles of length $2$.
Their result was further extended by R\"{o}dl and \v{S}i\v{n}ajov\'{a} \cite{RS94} to the larger family of $(n,k,\ell)$-systems.

\subsection{$(n,k,\ell)$-systems and $(n,k,\ell)$-omitting systems}
Let $k> \ell \ge 1$.
An $n$-vertex $k$-graph $\mathcal{H}$ is an
{\em $(n,k,\ell)$-system} if the intersection of every pair of edges in $\mathcal{H}$ has size less than $\ell$,
and $\mathcal{H}$ is an {\em $(n,k,\ell)$-omitting system}
if it has no two edges whose intersection has size exactly $\ell$.
It is clear from the definition that an $(n,k,\ell)$-system is an $(n,k,\ell)$-omitting system,
but not vice versa,
since an $(n,k,\ell)$-omitting system may have pairwise intersection sizes greater than $\ell$.

Define
\begin{align}
f(n,k,\ell) &= \min\left\{\alpha(\mathcal{H}): \mathcal{H} \text{ is an $(n,k,\ell)$-system} \right\}, \quad{\rm and}\quad\notag\\
g(n,k,\ell) &= \min\left\{\alpha(\mathcal{H}): \mathcal{H} \text{ is an $(n,k,\ell)$-omitting system} \right\}. \notag
\end{align}

We will use the standard asymptotic notations $O, \Omega, \Theta, o$ to simplify the formulas used in the present paper,
and the limit is generally taken with respect to $n$ unless noted otherwise.

The study of $f(n,k,\ell)$ has a long history (e.g. \cite{RS94,KMV14,EV13,TL18}) and, in particular,
R\"{o}dl and \v{S}i\v{n}ajov\'{a} \cite{RS94} proved that
\begin{align}\label{equ:Rodl-Sinajov-bound}
f(n,k,\ell) = \Theta\left(n^{\frac{k-\ell}{k-1}} (\log n)^{\frac{1}{k-1}}\right) \mbox{ for all fixed } k > \ell \ge 2.
\end{align}
It follows that
\begin{align}\label{equ:upper-bound-g-Rodl}
g(n,k,\ell) \le f(n,k,\ell) = O\left( n^{\frac{k-\ell}{k-1}} (\log n)^{\frac{1}{k-1}} \right).
\end{align}

One important difference between $(n,k,\ell)$-systems and $(n,k,\ell)$-omitting systems is their maximum sizes.
By definition,
every set of $\ell$ vertices in an $(n,k,\ell)$-system is contained in at most one edge, thus
every $(n,k,\ell)$-system has size at most $\binom{n}{\ell}/\binom{k}{\ell} = O\left( n^{\ell}\right)$.
However, this is not true for $(n,k,\ell)$-omitting systems.
Indeed, the following result of Frankl and F\"{u}redi \cite{FF85} shows that the maximum size of an $(n,k,\ell)$-omitting system
can be much larger than that of an $(n,k,\ell)$-system when $k> 2\ell+1$.

\begin{theorem}[Frankl-F\"{u}redi \cite{FF85}]\label{THM:FF85}
Let $k > \ell \ge 1$ be fixed integers and $\mathcal{H}$ be an $(n,k,\ell)$-omitting system.
Then $|\mathcal{H}| =O\left( n^{\max\{\ell,k-\ell-1\}}\right)$.
Moreover, the bound is tight up to a constant multiplicative factor.
\end{theorem}

Theorem~\ref{THM:FF85} together with $(\ref{equ:Spencer-bound})$ imply that for fixed $k,\ell$,
\begin{align}\label{equ:lower-bound-g-Spencer}
g(n,k,\ell) =
\begin{cases}
\Omega\left(n^{\frac{k-\ell}{k-1}}\right) & k \le 2\ell+1, \\
\Omega\left(n^{\frac{\ell+1}{k-1}}\right) & k > 2\ell+1.
\end{cases}
\end{align}

Notice that for $k\le 2\ell+1$ the bounds given by $(\ref{equ:upper-bound-g-Rodl})$ and $(\ref{equ:lower-bound-g-Spencer})$
match except for a factor of $(\log n)^{1/(k-1)}$,
but for $k>2\ell+1$, these two bounds have a gap in the exponent of $n$.

Our main goal in this paper is to extend the results of R\"{o}dl and \v{S}i\v{n}ajov\'{a}
to the larger class of $(n,k,\ell)$-omitting systems and improve the bounds given by
$(\ref{equ:upper-bound-g-Rodl})$ and $(\ref{equ:lower-bound-g-Spencer})$.
In other words, the question we focus on is the following:
\begin{align}
\text{What is the value of $g(n,k,\ell)$?} \notag
\end{align}

Our results for $(n,k,\ell)$-omitting systems are divided into two parts.
For $k\le 2\ell+1$, we believe that the behavior is similar to that of $(n,k,\ell)$-systems and
prove a nontrivial lower bound for the first open case $\ell=k-2$.
For $k>2\ell+1$ we give new lower and upper bounds which show that the minimum independence number of $(n,k,\ell)$-omitting systems has a very different behavior than for $(n,k,\ell)$-systems.

{\bf Remark.}
Let $k\ge 3$ and $L\subset \{0,1,\ldots,k-1\}$.
An $(n,k,L)$-omitting system is an $n$-vertex $k$-graph that has no two edges whose intersection
has size in $L$.
Our methods can be applied to the more general setting of $(n,k,L)$-omitting systems,
but in the present paper we only consider the case $|L| = 1$.

\subsection{$k \le 2\ell+1$}
As mentioned above, for this range of $\ell$ and $k$, the issue at hand is only the polylogarithmic factor in $g(n,k,\ell)$. It follows from the definition that an $(n,k,k-1)$-omitting system is also an $(n,k,k-1)$-system,
thus R\"{o}dl and \v{S}i\v{n}ajov\'{a}'s result $(\ref{equ:Rodl-Sinajov-bound})$ implies that
\begin{align}
g(n,k,k-1) = f(n,k,k-1) = \Theta\left( n^{\frac{1}{k-1}} (\log n)^{\frac{1}{k-1}} \right). \notag
\end{align}
So, the first open case in the range of $k\le 2\ell+1$ is $\ell=k-2$,
and for this case we prove the following nontrivial lower bound for $g(n,k,k-2)$,
which improves $(\ref{equ:lower-bound-g-Spencer})$.

\begin{theorem}\label{THM:indep-ell-k-2}
Suppose that $k \ge 4$.
Then every $(n,k,k-2)$-omitting system has an independent set of size
$\Omega\left(n^{2/(k-1)} \left(\log \log n\right)^{1/(k-1)}\right)$.
In other words,
\begin{align}
g(n,k,k-2) = \Omega\left(n^{\frac{2}{k-1}} \left(\log \log n\right)^{\frac{1}{k-1}}\right). \notag
\end{align}
\end{theorem}

Unfortunately,
our method for proving Theorem~\ref{THM:indep-ell-k-2} cannot be extended to the entire range of $k\le 2\ell+1$,
but we make the following conjecture.

\begin{conjecture}\label{CONJ:indep-ell-large}
For all fixed integers $k> \ell \ge 2$ that satisfy $k\le 2\ell+1$
there exists a function $\omega(n) \to \infty$ as $n \to \infty$
such that $g(n,k,\ell) = \Omega\left(n^{\frac{k-\ell}{k-1}}\omega(n)\right)$.
\end{conjecture}

Theorem \ref{THM:indep-ell-k-2} shows that Conjecture~\ref{CONJ:indep-ell-large} is true for $\ell=k-2$.
The smallest open case is $k=5$ and $\ell=2$.

\subsection{$k> 2\ell+1$}
Recall that in the range of $k>2\ell+1$ the bounds given by $(\ref{equ:upper-bound-g-Rodl})$ and $(\ref{equ:lower-bound-g-Spencer})$
leave a gap in the exponent of $n$.
The following result shows that for a wide range of $k$ and $\ell$
neither of them gives the correct order of magnitude.

\begin{theorem}\label{THM:indep-number-ell-small}
Let $\ell \ge 2$ and $k > 2\ell+1$ be fixed. Then
\begin{align}
\Omega\left(\max\left\{n^{\frac{\ell+1}{3\ell-1}}, n^{\frac{\ell+1}{k-1}}\right\}\right)
=  g(n,k,\ell) =
O\left(n^{\frac{\ell+1}{2\ell}} \left(\log n\right)^{\frac{1}{\ell}}\right). \notag
\end{align}
\end{theorem}

\textbf{Remark.}
\begin{itemize}
\item[(a)] The lower bound $n^{\frac{\ell+1}{3\ell-1}}$ can be improved to $n^{\frac{3-\sqrt{5}}{2}+o_{\ell}(1)} \sim n^{0.38196+o_{\ell}(1)}$.
See the remark in the end of Section~\ref{SEC:omit-system-ell-small} for details.
\item[(b)] It is clear that Theorem~\ref{THM:indep-number-ell-small} improves the bound given by $(\ref{equ:lower-bound-g-Spencer})$
for $k>3\ell$, and it also improves the bound given by $(\ref{equ:upper-bound-g-Rodl})$ for $k>2\ell+1$ as
$\frac{k-\ell}{k-1} - \frac{\ell+1}{2\ell} = \frac{(\ell-1)(k-2\ell-1)}{2\ell(k-1)} > 0$ for $k>2\ell+1$.
\end{itemize}

It would be interesting to determine $g(n,k,\ell)$ for $k>2\ell+1$. Here, we are not able to offer a conjecture for the exponent of $n$.

\begin{problem}\label{PROB:ell-small}
Determine the order of magnitude of $g(n,k,\ell)$ for $k>2\ell+1$.
\end{problem}

For the first open case  $(k,\ell)=(6,2)$ Theorem~\ref{THM:indep-number-ell-small} gives
$\Omega\left(n^{3/5}\right) = g(n,6,2) = O\left(n^{3/4+o(1)}\right)$.
Similar to Remark $(a)$ above the lower bound for $g(n,6,2)$ can be improved to $\Omega\left(n^{2/3}\right)$.
See the remark in the end of Section~\ref{SEC:omit-system-ell-small} for details.

\subsection{$(n,k,\ell,\lambda)$-systems and $(n,k,\ell,\lambda)$-omitting systems}
Let $k> \ell \ge 1$ and $\lambda\ge 1$ be integers.
The $k$-graph $S^k_{\lambda}(\ell)$ consists of ${\lambda}$ edges $E_1,\ldots,E_{\lambda}$ such that
$E_i \cap E_j = S$ for $1 \le i < j \le {\lambda}$ and some fixed set $S$ (called the center) of size $\ell$.
When $\ell=1$ we just write $S_{\lambda}^k$, and we will omit the superscript $k$ in $S^k_{\lambda}(\ell)$ if it is obvious.

It is easy to see that an $n$-vertex $k$-graph is an $(n,k,\ell)$-omitting system iff it is $S_{2}(\ell)$-free,
and is an $(n,k,\ell)$-system iff it is $\{S_{2}(\ell), \ldots, S_{2}(k-1)\}$-free.
This motivates us to define the following generalization of $(n,k,\ell)$-omitting systems and $(n,k,\ell)$-systems.

An $n$-vertex $k$-graph $\mathcal{H}$ is an
{\em $(n,k,\ell,\lambda)$-system} if every set of $\ell$ vertices is contained in at most $\lambda$ edges,
and $\mathcal{H}$ is an {\em $(n,k,\ell,\lambda)$-omitting system}
if it does not contain $S_{\lambda+1}(\ell)$ as a subgraph.

Define
\begin{align}
f(n,k,\ell,\lambda) &= \min\left\{\alpha(\mathcal{H}): \mathcal{H} \text{ is an $(n,k,\ell,\lambda)$-system} \right\},
\quad{\rm and}\quad\notag\\
g(n,k,\ell,\lambda) &= \min\left\{\alpha(\mathcal{H}): \mathcal{H} \text{ is an $(n,k,\ell,\lambda)$-omitting system} \right\}. \notag
\end{align}

When $\lambda$ is a fixed constant, the value of $f(n,k,\ell,\lambda)$ is essentially the same as $f(n,k,\ell)$ (e.g. see \cite{RS94}),
i.e. $f(n,k,\ell,\lambda) = \Theta\left(f(n,k,\ell)\right)$.
Similarly, the same conclusions as in Theorems~\ref{THM:indep-ell-k-2} and \ref{THM:indep-number-ell-small}
also hold for $g(n,k,\ell,\lambda)$, since Theorem~\ref{THM:FF85} holds for all $S_{\lambda}(\ell)$-free hypergraphs
and using it one can easily extend the proof for the case $\lambda = 1$ to the case $\lambda > 1$.
For the sake of simplicity, we will prove Theorem~\ref{THM:indep-ell-k-2} only for the case $\lambda =1$.

When $\lambda$ is not a constant, even the value of $f(n,k,\ell,\lambda)$ is not known in general.
Here is a summary of the known results.
\begin{itemize}
    \item $\ell=1$: An $(n,k,1, \lambda)$-system is just a $k$-graph with maximum degree $\lambda$ and here complete $k$-graphs and (\ref{equ:Spencer-bound}) yield
$$f(n,k,1, \lambda) = \Theta\left(\frac{n}{\lambda^{1/(k-1)}}\right).$$
On the other hand a result of Loh~\cite{Loh09} implies
 $$g(n,k,1,\lambda) = \frac{n}{\lambda+1}\qquad \hbox{whenever} \qquad  (\lambda+1)(k-1)\mid n.$$
If the divisibility condition fails then we have a small error term above.

\item $\ell=k-1$: Kostochka, Mubayi, and Verstra\"{e}te \cite{KMV14} proved that
\begin{align}
f(n,k,k-1,\lambda) = \Theta\left(\left(\frac{n}{\lambda}\right)^{\frac{1}{k-1}} \left(\log\frac{n}{\lambda}\right)^{\frac{1}{k-1}} \right)
\quad{\rm for}\quad 1\le \lambda \le \frac{n}{(\log n)^{3(k-1)^2}}. \notag
\end{align}

\item $2\le \ell\le k-2$: Tian and Liu \cite{TL18} proved that
\begin{align}
f(n,k,\ell,\lambda)
= \Omega\left(\left(\frac{n}{\lambda}\log \frac{n}{\lambda}\right)^{1/\ell}\right)
\quad {\rm for} \quad
k \ge 5,\mbox{ } \frac{2k+4}{5} < \ell \le k-2,\mbox{ } \lambda =o\left( n^{\frac{5\ell-2k-4}{3k-9}} \right). \notag
\end{align}
They also gave a construction which implies that
\begin{align}
f(n,k,\ell,\lambda)
=O\left( \left(\frac{n^{k-\ell}}{\lambda}\right)^{\frac{1}{k-1}}\left(\log \frac{n}{\lambda}\right)^{\frac{1}{k-1}}\right)
\quad{\rm for}\quad 2 \le \ell \le k-1,\mbox{ } \log n \ll \lambda \ll n. \notag
\end{align}
\end{itemize}

Since for every $\lambda>0$ an $(n,k,\ell,\lambda)$-system has size $O\left(\lambda n^{\ell}\right)$,
it follows from $(\ref{equ:Spencer-bound})$ that
\begin{align}
f(n,k,\ell,\lambda) = \Omega\left(\left(\frac{n^{k-\ell}}{\lambda}\right)^{\frac{1}{k-1}}\right), \notag
\end{align}
which, by Tian and Liu's upper bound, is tight up to a factor of $(\log n)^{1/(k-1)}$.

Using a result of Duke, Lefmann, and R\"{o}dl \cite{DLR95} we are able to improve the lower bound for $f(n,k,\ell,\lambda)$
to match the upper bound obtained by Tian and Liu for a wide range of $\lambda$.

\begin{theorem}\label{THM:independence-number-lambda-design}
Let $k>\ell \ge 2$ be fixed.
If there exists a constant $\delta>0$ such that $0 < \lambda < n^{\frac{\ell-1}{k-2}-\delta}$, then
\begin{align}
f(n,k,\ell,\lambda) = \Omega\left(\left(\frac{n^{k-\ell}}{\lambda}\right)^{\frac{1}{k-1}}\left(\log n\right)^{\frac{1}{k-1}} \right). \notag
\end{align}
\end{theorem}

{\bf Remark.}
It remains open to determine $f(n,k,\ell,\lambda)$ for
$\Omega\left(n^{\frac{\ell-1}{k-2}-o(1)}\right) = \lambda = O\left(n^{k-\ell}\right)$.

Since Theorem~\ref{THM:FF85} does not hold when $\lambda$ is not a constant,
our method of proving Theorems~\ref{THM:indep-ell-k-2} and \ref{THM:indep-number-ell-small}
cannot be extended to this case.

\subsection{Applications in Ramsey theory}
For a $k$-graph $\mathcal{F}$ the {\em Ramsey number $r_k(\mathcal{F},t)$} is the smallest integer $n$ such that every
$\mathcal{F}$-free $k$-graph on $n$ vertices has an independent set of size at least $t$.
Determining the minimum independence number of an $\mathcal{F}$-free $k$-graph on $n$ vertices
is essentially the same as determining the value of $r_k(\mathcal{F},t)$.
So, our results above can be applied to determine the Ramsey number of some hypergraphs.

First, Theorem~\ref{THM:indep-ell-k-2} and $(\ref{equ:upper-bound-g-Rodl})$ imply the following corollary.

\begin{corollary}\label{CORO:Ramsey-number-S-k-2}
Let $k \ge 4$ and $\lambda \ge 2$ be fixed integers.
Then 
\begin{align}
\Omega\left(\frac{t^{(k-1)/2}}{\left(\log t\right)^{1/2}}\right)
= r_{k}(S_{\lambda}(k-2),t)
= O\left(\frac{t^{(k-1)/2}}{\left(\log\log t\right)^{1/2}}\right). \notag
\end{align}
\end{corollary}

Similarly, Theorem~\ref{THM:indep-number-ell-small} gives the following corollary.

\begin{corollary}\label{CORO:Ramsey-number-S-ell-small}
Let $\ell\ge 2$, $k>2\ell+1$, and $\lambda \ge 2$ be fixed integers.
Then
\begin{align}
\Omega\left(\frac{t^{{2\ell}/{(\ell+1)}}}{\left(\log t\right)^{{2}/{(\ell+1)}}}\right)
= r_{k}(S_{\lambda}(\ell),t)
= O\left(\min\left\{t^{\frac{3\ell-1}{\ell+1}}, t^{\frac{k-1}{\ell+1}}\right\}\right). \notag
\end{align}
\end{corollary}

{\bf Remark.}
According to Remark $(a)$ after Theorem~\ref{THM:indep-number-ell-small}, the upper bound $t^{\frac{3\ell-1}{\ell+1}}$
above can be improved to $t^{\frac{3+\sqrt{5}}{2}+o_{\ell}(1)} \sim t^{2.61803+o_{\ell}(1)}$.

The following result about $r_{k}(S_{\lambda}^k,t)$ follows from a more general result of Loh~\cite{Loh09}.

\begin{theorem}[Loh~\cite{Loh09}]\label{THM:Ramsey-linear-tree}
Let $t \ge k \ge 2$, $t-1 = q(k-1)+r$ for some $q, r\in \mathbb{N}$ with $0 \le r \le k-2$.
Then for every $\lambda \ge 2$
$$\lambda q (k-1) + r + 1 \le r_{k}(S_{\lambda}^k,t) \le \lambda q(k-1)+ \lambda r + 1.$$
In particular, $r_{k}(S_{\lambda}^k,t) = \lambda(t-1)+1$ whenever $(k-1)\mid (t-1)$.
\end{theorem}

The $k$-$Fan$, denoted by $F^k$, is the $k$-graph consisting of $k+1$ edges $E_1,\ldots,E_k,E$
such that $E_{i} \cap E_j = v$ for all $1\le i < j \le k$,
where $v \not\in E$, and $|E_i \cap E| = 1$ for $1\le i \le k$.
In other words, $F^k$ is obtained from $S_{k}^{k}$ by adding an edge omitting $v$ that intersects  each edge of $S_{k}^{k}$.
It is easy to see that $F^2$ is just the triangle $K_3$.
The $k$-graph $F^k$ was first introduced by Mubayi and Pikhurko \cite{MP07} in order to extend Mantel's theorem to hypergraphs.
Unlike the case  $k=2$, where it is well known that
$r_2(K_3, t) = \Theta\left({t^2}/{\log t}\right)$ (e.g. see \cite{AKS80,Kim95}),
the following result shows that $r_k(F^k, t) = \Theta(t^2)$ for all $k\ge 3$.

\begin{theorem}\label{THM:Ramsey-k-fan}
Suppose that $t \ge k \ge 3$.
Then
\begin{align}
\left\lfloor \frac{t}{2} \right\rfloor \left\lfloor \frac{t-1}{2(k-2)} \right\rfloor < r_{k}(F^{k},t) \le t(t-1)+1. \notag
\end{align}
\end{theorem}

As $t \rightarrow \infty$, it remains open to determine $\lim r_k(F^k, t)/t^2$.

In Section~\ref{SEC:omit-system-ell-big}, we prove Theorem~\ref{THM:indep-ell-k-2}.
In Section~\ref{SEC:omit-system-ell-small}, we prove Theorem~\ref{THM:indep-number-ell-small}.
In Section~\ref{SEC:indepence-number-design}, we prove Theorem~\ref{THM:independence-number-lambda-design}.
In Section~\ref{SEC:linear-tree}, we prove Theorem~\ref{THM:Ramsey-k-fan}.

\section{Proof of Theorem~\ref{THM:indep-ell-k-2}}\label{SEC:omit-system-ell-big}
In this section we prove Theorem~\ref{THM:indep-ell-k-2}.
Let us show some preliminary results first.

\subsection{Preliminaries}
For a $k$-graph $\mathcal{H}$ and $i\in[k-1]$ the {\em $i$-th shadow} of $\mathcal{H}$ is
\begin{align}
\partial_{i}\mathcal{H} = \left\{A\in \binom{V(\mathcal{H})}{k-i}: \exists E\in \mathcal{H} \text{ such that } A\subset E \right\}. \notag
\end{align}
The {\em shadow} of $\mathcal{H}$ is $\partial\mathcal{H} = \partial_{1}\mathcal{H}$.
For a set $S \subset V(\mathcal{H})$ the {\em neighborhood} of $S$ in $\mathcal{H}$ is
\begin{align}
N_{\mathcal{H}}(S) = \left\{v\in V(\mathcal{H})\setminus S: \exists E\in \mathcal{H} \text{ such that } S\cup \{v\}\subset E\right\}, \notag
\end{align}
the {\em link} of $S$ in $\mathcal{H}$ is
\begin{align}
L_{\mathcal{H}}(S) = \left\{E\setminus S: E \in \mathcal{H} \text{ and } S\subset E\right\}, \notag
\end{align}
and $d_{\mathcal{H}}(S) = |L_{\mathcal{H}}(S)|$ is the {\em degree} of $S$ in $\mathcal{H}$.
For $i\in[k-1]$ the {\em maximum $i$-degree} of $\mathcal{H}$ is
\begin{align}
\Delta_i(\mathcal{H}) = \max\left\{d_{\mathcal{H}}(A): A\in \binom{V(\mathcal{H})}{i}\right\}, \notag
\end{align}
and note that $\Delta(\mathcal{H}) = \Delta_{1}(\mathcal{H})$.

For a pair of distinct vertices $u,v \in V(\mathcal{H})$ the {\em $(k-1)$-codegree} of $u$ and $v$
is the number of $(k-1)$-sets $S\subset V(\mathcal{H})$
such that $S\cup \{u\} \in \mathcal{H}$ and $S\cup \{v\} \in \mathcal{H}$.
Denoted by $\Gamma(\mathcal{H})$ the maximum $(k-1)$-codegree of $\mathcal{H}$.

\textbf{The random greedy independent set algorithm.}
We begin with $\mathcal{H}(0) = \mathcal{H}, V(0) = V(\mathcal{H})$ and $I(0) = \emptyset$.
Given independent set $I(i)$ and hypergraph $\mathcal{H}(i)$ on vertex set $V(i)$,
a vertex $v \in V(i)$ is chosen uniformly at random and added to $I(i)$ to form $I(i+1)$.
The vertex set $V(i+1)$ is set equal to $V(i)$ less $v$ and all vertices $u$ such that $\{u,v\}$ is an edge in $\mathcal{H}(i)$.
The hypergraph $\mathcal{H}(i+1)$ is formed form $\mathcal{H}_{i}$ by
\begin{itemize}
\item[1.] removing $v$ from all edges of size at least three in $\mathcal{H}(i)$ that contain $v$, and
\item[2.] removing every edge that contains a vertex $u$ such that the pair $\{u,v\}$ is an edge of $\mathcal{H}(i)$.
\end{itemize}
The process terminates when $V(i) = \emptyset$. At this point $I(i)$ is a maximal independent set in $\mathcal{H}$.
Let $i_{\max}$ denote the step where the algorithm terminates.

In \cite{BB16}, Bennett and Bohman analyzed the random greedy independent set algorithm
using the differential equation method,
and they proved that if a $k$-graph satisfies certain degree and codegree conditions,
then the random greedy independent set algorithm produces a large independent set with high probability.

\begin{theorem}[Bennett-Bohman \cite{BB16}]\label{THM:BB16-A}
Let $k$ and $\epsilon > 0$ be fixed.
Let $\mathcal{H}$ be a $D$-regular $k$-graph on $n$ vertices such that $D > n^{\epsilon}$.
If
\begin{align}
\Delta_{i}(\mathcal{H}) < D^{\frac{k-i}{k-1}-\epsilon}\quad{\rm for}\quad 2 \le i \le k-1,
\quad{\rm and}\quad
\Gamma(\mathcal{H}) < D^{1-\epsilon}, \notag
\end{align}
then the random greedy independent set algorithm produces an independent set $I$ in $\mathcal{H}$
of size $\Omega\left( \left({\log n}\right)^{{1}/{(k-1)}} \cdot n/D^{1/(k-1)}\right)$
with probability $1-o(1)$.
\end{theorem}

The lower bound on independence number in Theorem~\ref{THM:BB16-A} can easily be proved by applying a theorem of Duke-Lefmann-R\"odl~\cite{DLR95} (see Theorem~\ref{THM:DLR-Uncrowd}), so the main novelty of Theorem~\ref{THM:BB16-A} is the fact that the random greedy independent set algorithm produces an independent set of this size with high probability.

Let $S\subset V(\mathcal{H})$ be a set of bounded size $s$ such that $S$ contains no edge in $\mathcal{H}$.
A nice property of the random greedy independent set algorithm is that
$S$ is contained in the set $I(i)$ with probability $(1+o(1))\left(i/n\right)^{s}$,
which is almost the probability that $S$ is contained in a random $i$-subset of $V(\mathcal{H})$.

Using this property we can easily control the size of the induced subgraph of $\mathcal{G}$ on $I(i)$,
where $\mathcal{G}$ is a hypergraph that has the same vertex set with $\mathcal{H}$.

\begin{proposition}[Bennett-Bohman \cite{BB16}]\label{PROP:BB16}
Let $\mathcal{H}$ be a hypergraph that satisfies the conditions in Theorem~\ref{THM:BB16-A} and
$\mathcal{G}$ be a $k'$-graph on $V(\mathcal{H})$ (i.e. $\mathcal{G}$ and $\mathcal{H}$ are on the same vertex set).
If $i \le i_{\max}$ is fixed,
then the expected number of edges of $\mathcal{G}$ contained in $I(i)$ is at most $(1+o(1))\left(i/n\right)^{k'}\cdot |\mathcal{G}|$.
\end{proposition}

For $2\le j \le k-1$ and two edges $E,E'$ in a $k$-graph $\mathcal{H}$ we say $\{E,E'\}$ is a {\em $(2,j)$-cycle}
if $|E\cap E'| = j$.
Denote by $C_{\mathcal{H}}(2,j)$ the number of $(2,j)$-cycles in $\mathcal{H}$.
A hypergraph is {\em linear} if every pair of edges has at most one vertex in common.
It is easy to see that $\mathcal{H}$ is linear iff $C_{\mathcal{H}}(2,j) = 0$ for $2\le j \le k-1$.
The next theorem on the independence number of linear hypergraphs is due to Frieze and Mubayi \cite{FM13}.

\begin{theorem}[Frieze-Mubayi \cite{FM13}]\label{THM:FM13}
Suppose that $\mathcal{H}$ is a linear $k$-graph with $n$ vertices and average degree $d$.
Then $\alpha(\mathcal{H}) = \Omega\left(\left(\log d\right)^{1/(k-1)}\cdot {n}/{d^{1/(k-1)}}\right)$.
\end{theorem}

For a (not necessarily uniform) hypergraph $\mathcal{H}$ on $n$ vertices (assuming that $V(\mathcal{H}) = [n]$)
and a family $\mathcal{F} = \{\mathcal{G}_1,\ldots,\mathcal{G}_n\}$ of $m$-vertex $k$-graphs with
$V(\mathcal{G}_1) = \cdots = V(\mathcal{G}_n) = V_{\mathcal{F}}$
the {\em Cartesian product} of $\mathcal{H}$ and $\mathcal{F}$, denoted by $\mathcal{H}\Box\mathcal{F}$,
is a hypergraph on $V(\mathcal{H})\times V_{\mathcal{F}}$ and
\begin{align}
\mathcal{H}\Box\mathcal{F}
= \left\{(E,v): E\in \mathcal{H} \text{ and } v\in V_{\mathcal{F}}\right\}
   \cup \left\{(i,F): i\in[n] \text{ and } F\in \mathcal{G}_i\right\}. \notag
\end{align}

Since the hypergraphs we considered here are not necessarily regular,
Theorem~\ref{THM:BB16-A} cannot be applied directly to our situations.
To overcome this issue we use an adaption of a trick used by Shearer in \cite{SH95},
that is, for every nonregular hypergraph $\mathcal{H}$ we take the Cartesian product
of $\mathcal{H}$ and a family of linear hypergraphs to get a new hypergraph $\widehat{\mathcal{H}}$ that is regular.
Then we apply Theorem~\ref{THM:BB16-A} to $\widehat{\mathcal{H}}$ to get a large independent set,
and by the Pigeonhole principle, this ensures that $\mathcal{H}$ has a large independent set.

First, we need the following theorem to show the existence of sparse regular linear hypergraphs.

Given two $k$-graphs $\mathcal{H}_1$ and $\mathcal{H}_2$ with the same number of vertices
a {\em packing} of $\mathcal{H}_1$ and $\mathcal{H}_2$ is a bijection $\phi: V(\mathcal{H}_1) \to V(\mathcal{H}_2)$
such that $\phi(E)\not\in \mathcal{H}_2$ for all $E\in \mathcal{H}_1$.

\begin{theorem}[Lu-Sz\'{e}kely \cite{LS07}]\label{THEOREM:packing-hypergraphs}
Let $\mathcal{H}_1$ and $\mathcal{H}_2$ be two $k$-graphs on $n$ vertices.
If
\begin{align}
\Delta(\mathcal{H}_1) |\mathcal{H}_2| + \Delta(\mathcal{H}_2) |\mathcal{H}_1| < \frac{1}{ek} \binom{n}{k}, \notag
\end{align}
then there is a packing of $\mathcal{H}_1$ and $\mathcal{H}_2$.
\end{theorem}

Theorem~\ref{THEOREM:packing-hypergraphs} enables us to construct sparse regular linear hypergraphs inductively.

\begin{lemma}\label{LEMMA:exist-regular-linear-r-gp}
For every positive integer $n$ that satisfies $k\mid n$ and every positive integer $d$ that satisfies
\begin{align}
d \le \frac{(n-k+2)(n-k+1)}{e k^2(k-1)^2n} + 1,  \notag
\end{align}
there exists a $d$-regular linear $k$-graph with $n$ vertices.
\end{lemma}
\begin{proof}[Proof of Lemma \ref{LEMMA:exist-regular-linear-r-gp}]
We proceed by induction on $d$ and note that the case $d = 1$ is trivial since
a perfect matching on $n$ vertices is a $1$-regular linear $k$-graph.
Now suppose that $d\ge 2$.
By the induction hypothesis, there exists a $(d-1)$-regular linear $k$-graph on $n$ vertices,
and let $\mathcal{H}_{d-1}$ be such a $k$-graph.
Let $\mathcal{H}_{1}$ be a perfect matching on $n$ vertices.
Define the extended $k$-graph $\widehat{\mathcal{H}}_{1}$ of $\mathcal{H}_1$ as
\begin{align}
\widehat{\mathcal{H}}_{1}
= \left\{\{u,v\}\cup A:
\{u,v\}\in \partial_{k-2}\mathcal{H}_1 \text{ and } A\in \binom{V(\mathcal{H}_1)\setminus\{u,v\}}{k-2} \right\}.\notag
\end{align}
It is clear from the definition that $\mathcal{H}_1 \subset \widehat{\mathcal{H}}_{1}$,
$|\widehat{\mathcal{H}}_{1}| < \frac{n}{k}\binom{k}{2}\binom{n}{k-2}$, and
$\widehat{\mathcal{H}}_{1}$ is regular.
So,
\begin{align}
\Delta(\widehat{\mathcal{H}}_{1})
= \frac{k|\widehat{\mathcal{H}}_{1}|}{n}
< \frac{k}{n}\frac{n}{k}\binom{k}{2}\binom{n}{k-2} =\binom{k}{2}\binom{n}{k-2}. \notag
\end{align}
By assumption
\begin{align}
\Delta(\mathcal{H}_{d-1}) |\widehat{\mathcal{H}}_1| + \Delta(\widehat{\mathcal{H}}_1) |\mathcal{H}_{d-1}|
& < (d-1)\frac{n}{k}\binom{k}{2}\binom{n}{k-2} + \frac{(d-1)n}{k}\binom{k}{2}\binom{n}{k-2} \notag\\
& = 2(d-1)\frac{n}{k}\binom{k}{2}\binom{n}{k-2}
  \le \frac{1}{ek}\binom{n}{k}. \notag
\end{align}
Therefore, by Theorem \ref{THEOREM:packing-hypergraphs}, there exist a bijection
$\phi: V(\mathcal{H}_{d-1}) \to V(\mathcal{H}_{1})$ such that
$|\phi(E) \cap E'| \le k-1$ for all $E\in \mathcal{H}_{d-1}$ and all $E'\in \widehat{\mathcal{H}}_{1}$,
and this implies that $|\phi(E) \cap E''| \le 1$ for all $E\in \mathcal{H}_{d-1}$ and all $E''\in \mathcal{H}_{1}$.
Therefore, $\mathcal{H}_1 \cup \phi\left(\mathcal{H}_{d-1}\right)$ is a $d$-regular linear $k$-graph on $n$ vertices.
\end{proof}

\subsection{Proofs}
First we use Theorem~\ref{THM:BB16-A} and Proposition~\ref{PROP:BB16} to prove a result about
the common independent set of two hypergraphs on the same vertex set.

\begin{theorem}\label{THM:common-indep-set}
Let $k_1,k_2 \ge 2$ be integers, $\epsilon>0$, $n, D\in \mathbb{N}$, and $d>0$.
Suppose that
\begin{itemize}
\item[(a)] $\mathcal{H}$ is an $n$-vertex $k_1$-graph, $\mathcal{G}$ is an $n$-vertex $k_2$-graph, and
$V(\mathcal{H}) = V(\mathcal{G}) = V$,
\item[(b)] $D>n^{\epsilon}$ and
$d\left({\log n}/{D}\right)^{\frac{k_2-1}{k_1-1}} \gg 1$,
\item[(c)] $\mathcal{H}$ satisfies that $\Delta(\mathcal{H})\le D$,
\begin{align}
\Delta_{i}(\mathcal{H}) < D^{\frac{k_1-i}{k_1-1}-\epsilon}\quad{\rm for}\quad 2 \le i \le k_1-1,
\quad{\rm and}\quad
\Gamma(\mathcal{H}) < D^{1-\epsilon}, \notag
\end{align}
\item[(d)] $\mathcal{G}$ satisfies that $d(\mathcal{G}) \le d$ and
\begin{align}
C_{\mathcal{G}}(2,i) \ll n\left({D}/{\log n}\right)^{\frac{2k_2-i-1}{k_1-1}}
\quad{\rm for}\quad 2\le i \le k_2-1. \notag
\end{align}
\end{itemize}
Then, $\alpha\left(\mathcal{H}\cup \mathcal{G}\right) = \Omega\left(\omega \cdot {n}/{d^{1/(k_2-1)}}\right)$,
where
\begin{align}
\omega = \omega(n,D,d,k_1,k_2) = \left(\log\left(\left({\log n}/{D}\right)^{\frac{k_2-1}{k_1-1}}d\right)\right)^{1/(k_2-1)}. \notag
\end{align}
\end{theorem}

{\bf Remarks.}
\begin{itemize}
\item Although Theorem~\ref{THM:common-indep-set} imposes no condition on $k_1$ and $k_2$, we will only apply the result in the case $k_2=k_1+1$.

\item Spencer's bound $(\ref{equ:Spencer-bound})$ implies that
$\alpha(\mathcal{G}) = \Omega\left({n}/{d^{1/(k_2-1)}}\right)$.
Theorem~\ref{THM:common-indep-set} improves it in two ways:
first it improves the bound by a factor of $\omega$,
second it is a lower bound for the independence number of $\mathcal{G} \cup \mathcal{H}$.
Ajtai, Koml\'{o}s, Pintz, Spencer, and Szemer\'{e}di's result $(\ref{equ:AKPSS-bound})$ implies that
the upper bound for $\omega$ is $(\log n)^{1/(k_2-1)}$.
However, we are not able to show that $\omega = \Omega\left((\log n)^{1/(k_2-1)}\right)$ in general,
and it would be interesting to determine the optimal value of $\omega$.
\item If $\mathcal{H}$ and $\mathcal{G}$ satisfy conditions $(a)$ and $(c)$ in Theorem~\ref{THM:common-indep-set}
and also satisfy
\subitem $(b')$ $D>n^{\epsilon}$ and $d\left({\log n}/{D}\right)^{\frac{k_2-1}{k_1-1}} \ll 1$,

then $\alpha\left(\mathcal{H}\cup \mathcal{G}\right) = \Omega\left((\log n)^{1/(k_1-1)}\cdot {n}/{D^{1/(k_1-1)}}\right)$.
Moreover, if $\mathcal{G} = \emptyset$, then $\alpha(\mathcal{H}) = \Omega\left((\log n)^{1/(k_1-1)}\cdot {n}/{D^{1/(k_1-1)}}\right)$ which is the bound in Theorem~\ref{THM:BB16-A}.
The proof is similar to the proof of Theorem~\ref{THM:common-indep-set}.
\end{itemize}

\begin{proof}[Proof of Theorem~\ref{THM:common-indep-set}]
For $2\le i \le k_2-1$ define
\begin{align}
\mathcal{G}^{i} = \left\{S\in \binom{V}{2k_2-i}: \mathcal{G}[S] \text{ contains a $(2,i)$-cycle}\right\}. \notag
\end{align}
Fix $m\in \mathbb{N}$ such that $D \ll m = O(n^{k_1})$, and $k_1\mid m$.
Notice that $D$ has a trivial upper bound $n^{k_1-1}$, so such an integer $m$ exists.
For every $v\in V$ let $D_v = D-d_{\mathcal{H}}(v)$.
Since $m \gg D$ and $k_1\mid m$, by Lemma~\ref{LEMMA:exist-regular-linear-r-gp},
there exists a $D_v$-regular linear $k_1$-graph $\mathcal{F}(v)$ on $[m]$ for every $v\in V$.
Let
\begin{align}
\mathcal{H}' = \mathcal{H} \cup \mathcal{G}\cup \left(\bigcup_{2\le i \le k_2-1}\mathcal{G}^{i}\right), \quad
\mathcal{F} = \{\mathcal{F}(v): v\in V\}, \quad{\rm and}\quad
\widehat{\mathcal{H}}' = \mathcal{H}' \Box \mathcal{F}. \notag
\end{align}
Note that $\widehat{\mathcal{H}}'$ is consisting of
\begin{itemize}
\item[1.] the $k_1$-graph $\widehat{\mathcal{H}} = \mathcal{H} \Box \mathcal{F}$,
\item[2.] the $k_2$-graph $\widehat{\mathcal{G}}$ that is the union of $m$ pairwise vertex-disjoint copies of $\mathcal{G}$, and
\item[3.] the $(2k_2-i)$-graph $\widehat{\mathcal{G}}^{i}$ that is the union of $m$ pairwise vertex-disjoint
copies of $\mathcal{G}^{i}$ for $2\le i \le k_2-1$.
\end{itemize}

For every $v\in V(\widehat{\mathcal{H}})$ we have $d_{\widehat{\mathcal{H}}}(v) = d_{\mathcal{H}}(v)+D_v = D$.
Moreover,
\begin{align}
\Delta_i(\widehat{\mathcal{H}}) = \Delta_{i}(\mathcal{H}) < D^{\frac{k_1-i}{k_1-1}-\epsilon}\quad{\rm for}\quad 2 \le i \le k_1-1,
\quad{\rm and}\quad
\Gamma(\widehat{\mathcal{H}}) = \Gamma(\mathcal{H}) < D^{1-\epsilon}, \notag
\end{align}
Applying the random greedy independent set algorithm and Theorem~\ref{THM:BB16-A} to $\widehat{\mathcal{H}}$ ,
we obtain an independent set
$\hat{I}$ of size at least $c \left(\log nm\right)^{1/(k_1-1)}\cdot nm/D^{1/(k_1-1)}$
for some constant $c>0$ with probability $1-o(1)$.
Let $p = c\left((\log nm)/D\right)^{1/(k_1-1)}$ and
we may assume that $|\hat{I}| = pnm$ since otherwise we can take
the set of the first $pnm$ vertices generated by the random greedy independent set algorithm instead.

Applying Proposition~\ref{PROP:BB16} to $\widehat{\mathcal{G}},\widehat{\mathcal{G}}^1,\ldots, \widehat{\mathcal{G}}^{k_2-1}$
and by assumption $(d)$ we obtain
\begin{align}
\mathbb{E}\left[\left|\widehat{\mathcal{G}}[\hat{I}]\right|\right]
\le (1+o(1))p^{k_2}|\widehat{\mathcal{G}}|
< 2dnm p^{k_2}, \notag
\end{align}
and for $2\le i \le k_2-1$
\begin{align}
\mathbb{E}\left[\left|\widehat{\mathcal{G}}^{i}[\hat{I}]\right|\right]
= (1+o(1))p^{2k_2-i} \times m \times C_{\mathcal{G}}(2,i)
= o(pnm). \notag
\end{align}
So, by Markov's inequality and the union bound, with probability at least $1/2$ both
\begin{align}
\left|\widehat{\mathcal{G}}[\hat{I}]\right| \le 10 dnm p^{k_2}
\quad{\rm and}\quad
\left|\widehat{\mathcal{G}}^{i}[\hat{I}]\right| = o(pnm) \quad \forall \mbox{ } 2\le i \le k_2-1 \notag
\end{align}
hold.

Fix a set $\hat{I}$ such that $|\hat{I}| = pnm$ and the events above hold.
Then by removing $o(pnm)$ vertices we obtain a subset $\hat{I}'\subset \hat{I}$ such that
\begin{align}
\left|\widehat{\mathcal{G}}^{i}[\hat{I}]\right| = 0 \quad{\rm for}\quad 2\le i \le k_2-1. \notag
\end{align}
In other words, the $k_2$-graph $\widehat{\mathcal{G}}[\hat{I}']$ is linear.
Since
\begin{align}
d\left(\widehat{\mathcal{G}}[\hat{I}']\right)
\le \frac{k_2 \times 10 dnm p^{k_2}}{(1-o(1))pnm} \le 20k_2 d p^{k_2-1}, \notag
\end{align}
by Theorem~\ref{THM:FM13}, it has an independent set $I'$
of size at least
\begin{align}
\Omega\left(\frac{pnm}{(20 k_2 d p^{k_2-1})^{1/(k_2-1)}}\left(\log 20dp^{k_2-1}\right)^{\frac{1}{k_2-1}}\right)
& = \Omega\left(m\frac{n}{d^{1/(k_2-1)}}\left(\log p^{k_2-1}d\right)^{\frac{1}{k_2-1}}\right)\notag\\
& = \Omega\left(m\frac{n}{d^{1/(k_2-1)}}\omega\right). \notag
\end{align}
Here we used assumption $(b)$ to ensure that $20 k_2 d p^{k_2-1} \ge 1$.

By the Pigeonhole principle, there exists $j\in[m]$ such that $I = I' \cap (V\times \{j\})$ has size
at least $|\hat{I}|/m = \Omega\left(\omega \cdot {n}/{d^{1/(k_2-1)}}\right)$,
and it is clear that $I$ is an independent set in both $\mathcal{H}$ and $\mathcal{G}$.
\end{proof}

Next we use Theorem~\ref{THM:common-indep-set} to prove Theorem \ref{THM:indep-ell-k-2}.
The idea is to first decompose an $(n,k,k-2)$-omitting system $\mathcal{H}$
into two parts: $\mathcal{H}_{k-1}\subset \partial\mathcal{H}$ and $\mathcal{H}_{k} \subset \mathcal{H}$,
and then apply Theorem~\ref{THM:common-indep-set} to $\mathcal{H}_{k-1}$ and $\mathcal{H}_{k}$ to find a large
set $I \subset V$ that is independent in both of them.
It will be easy to see that the set $I$ is independent in $\mathcal{H}$.

\begin{proof}[Proof of Theorem \ref{THM:indep-ell-k-2}]
Let $\mathcal{H}$ be an $(n,k,k-2)$-omitting system and let $V = V(\mathcal{H})$.
By Theorem \ref{THM:FF85}, there exists a constant $C_1$ such that $|\mathcal{H}| \le C_1 n^{k-2}$.
Let $\beta = \beta(k)>0$ be a constant such that $\frac{k}{2(k-1)}<\beta<1$, for example, take $\beta = 4/5$.
Define
\begin{align}
\mathcal{H}_{k-1} = \left\{A\in \partial \mathcal{H}: d_{\mathcal{H}}(A) \ge \frac{n^{\frac{k-3}{k-1}}}{(\log n)^{\beta}} \right\}
\quad{\rm and}\quad
\mathcal{H}_{k} = \left\{E\in \mathcal{H}: \binom{E}{k-1} \cap \mathcal{H}_{k-1} = \emptyset \right\}. \notag
\end{align}

Let $k_1 = k-1$, $k_2 = k$, $D = n^{k-4+{2}/{(k-1)}}(\log n)^{\beta}$, $d = C_1 n^{k-3}$, and
$\epsilon$ be a constant such that $0<\epsilon<1/(k-1)$.
Then $D>n^{\epsilon}$ and
\begin{align}
d\left(\frac{\log n}{D}\right)^{\frac{k_2-1}{k_1-1}}
= C_1 n^{k-3} \left(\frac{\log n}{n^{k-4+\frac{2}{k-1}}(\log n)^{\beta}}\right)^{\frac{k-1}{k-2}}
= C_1 \left(\log n\right)^{(1-\beta)\frac{k-1}{k-2}}
\gg 1. \notag
\end{align}
Therefore, condition $(b)$ in Theorem~\ref{THM:common-indep-set} is satisfied.
Next we show that $\mathcal{H}_{k-1}$ and $\mathcal{H}_{k}$ satisfy $(c)$ and $(d)$ in Theorem~\ref{THM:common-indep-set}
with our choice of $k_1,k_2,D,d,\epsilon$.

\begin{claim}\label{CLAIM:H-k-1}
The $(k-1)$-graph $\mathcal{H}_{k-1}$ is an $\left(n,k-1,k-3,\lfloor n^{2/(k-1)}(\log n)^{\beta}\rfloor\right)$-system
and an $(n,k-1,k-2)$-system.
\end{claim}
\begin{proof}[Proof of Claim \ref{CLAIM:H-k-1}]
First we prove that $\mathcal{H}_{k-1}$ is an $(n,k-1,k-2)$-system.
Indeed, suppose to the contrary that there exist $e_1,e_2\in \mathcal{H}_{k-1}$ such that
$S = e_1\cap e_2$ has size $k-2$.
By the definition of $\mathcal{H}_{k-1}$,
$|N_{\mathcal{H}}(e_i)| \ge {n^{\frac{k-3}{k-1}}}/(\log n)^{\beta} > 2k$ for $i=1,2$.
So there exist $v_1,v_2\in V\setminus(e_1\cup e_2)$ such that
$E_i = e_i\cup \{v_i\} \in \mathcal{H}$ for $i=1,2$.
However, $E_1\cap E_2 =S$ has size $k-2$, contradicting the assumption that $\mathcal{H}$ is an $(n,k,k-2)$-omitting system.
Therefore, $\mathcal{H}_{k-1}$ is an $(n,k-1,k-2)$-system.

Now suppose to the contrary that there exists a set $A \subset V$ of size $k-3$
with $d_{\mathcal{H}_{k-1}}(A) = m >  n^{2/(k-1)}(\log n)^{\beta}$.
Since $\mathcal{H}_{k-1}$ is an $(n,k-1,k-2)$-system,
$L_{\mathcal{H}_{k-1}}(A)$ is a matching consisting of $m$ edges.
Suppose that $L_{\mathcal{H}_{k-1}}(A) = \{e_1,\ldots,e_m\}$, and
let $B_i = A \cup e_i$ for $1\le i \le m$.
Since $B_i \in \mathcal{H}_{k-1}$,
by definition, there exists
a set $N_i \subset V$ of size at least ${n^{\frac{k-3}{k-1}}}/(\log n)^{\beta}$
such that $B_i\cup \{u\} \in \mathcal{H}$ for all $u \in N_i$.

Suppose that there exists $v\in N_i\cap N_j$ for some distinct $i,j\in[m]$.
Then the two sets $A\cup e_i \cup \{v\}$ and $A\cup e_j \cup \{v\}$ are edges in $\mathcal{H}$
and have an intersection of size $k-2$, a contradiction.
Therefore, $N_i\cap N_j = \emptyset$ for all distinct $i,j\in[m]$.
It follows that
\begin{align}
n = |V| \ge \sum_{i\in[m]}|N_i|
\ge m {n^{\frac{k-3}{k-1}}}/(\log n)^{\beta}
> n^{\frac{2}{k}}(\log n)^{\beta} {n^{\frac{k-3}{k-1}}}/(\log n)^{\beta}
>n, \notag
\end{align}
a contradiction.
Therefore, $\Delta_{k-3}(\mathcal{H}_{k-1}) \le  n^{2/(k-1)}(\log n)^{\beta}$, which implies that
$\mathcal{H}_{k-1}$ is an $\left(n,k-1,k-3, \lfloor n^{2/(k-1)}(\log n)^{\beta}\rfloor\right)$-system.
\end{proof}

Since $\mathcal{H}_{k-1}$ is an
$\left(n,k-1,k-3, \lfloor n^{2/(k-1)}(\log n)^{\beta}\rfloor \right)$-system,
$\Delta_{k-3}(\mathcal{H}_{k-1}) \le  n^{2/(k-1)}(\log n)^{\beta}$.
Moreover, for every set $S \subset V$ of size $i$ with $i\in[k-4]$
the link $L_{\mathcal{H}_{k-1}}(S)$ is an $\left(n,k-1-i,k-3-i, n^{2/(k-1)}(\log n)^{\beta}\right)$-system.
Therefore, for $i \in [k-4]$
\begin{align}
\Delta_{i}(\mathcal{H}_{k-1})
\le n^{\frac{2}{k-1}}(\log n)^{\beta} \binom{n}{k-3-i}/\binom{k-1-i}{k-3-i} < n^{k-3-i+\frac{2}{k-1}}(\log n)^{\beta}. \notag
\end{align}
Since
\begin{align}
\left(k-4+\frac{2}{k-1}\right)\frac{k-1-i}{k-2} - \left(k-3-i+\frac{2}{k-1}\right)
= \frac{2(i-1)}{k-1}>\epsilon, \notag
\end{align}
we obtain
\begin{align}
\Delta_i(\mathcal{H}_{k-1})
< n^{k-3-i+\frac{2}{k-1}} (\log n)^{\beta}
< D^{\frac{k-1-i}{k-1-1}-\epsilon} \quad{\rm for}\quad 2\le i \le k-3. \notag
\end{align}
On the other hand, since $\mathcal{H}$ is an $(n,k-1,k-2)$-system,
$\Delta_{k-2}(\mathcal{H}_{k-1}) \le 1<  D^{\frac{k-1-(k-2)}{k-1-1}-\epsilon}$ and
$\Gamma(\mathcal{H}_{k-1}) = 0 < D^{1-\epsilon}$.
Therefore, $\mathcal{H}_{k-1}$ satisfies condition $(c)$ in Theorem~\ref{THM:common-indep-set}.

\begin{claim}\label{CLAIM:H-k}
The $k$-graph $\mathcal{H}_k$ satisfies $d(\mathcal{H}_k) \le C_1 k n^{k-2}$,
\begin{align}
C_{\mathcal{H}_k}(2,i) = O\left(n^{2k-4-i}\right) \quad{\rm for}\quad 2\le i \le k-3, \notag
\end{align}
$C_{\mathcal{H}_k}(2,k-2) = 0$, and $C_{\mathcal{H}_k}(2,k-1)= O\left(n^{k-2+\frac{k-3}{k-1}}/(\log n)^{\beta}\right)$.
\end{claim}
\begin{proof}[Proof of Claim~\ref{CLAIM:H-k}]
First, it is clear that $C_{\mathcal{H}_k}(2,k-2) = 0$ since there is no pair of edges in $\mathcal{H}_k$
with an intersection of size $k-2$.

Let $2\le i \le k-3$ and $S\subset V$ be a set of size $i$.
Since $\mathcal{H}_k$ is an $(n,k,k-2)$-omitting system,
the link $L_{\mathcal{H}_k}(S)$ is an $(n,k-i,k-2-i)$-omitting system.
So, by Theorem~\ref{THM:FF85}, $|L_{\mathcal{H}_k}(S)| = O\left(n^{k-2-i}\right)$,
which implies that
\begin{align}
C_{\mathcal{H}_k}(2,i)
\le |\mathcal{H}_k| \times \binom{k}{i} \times O\left(n^{k-2-i}\right)
= O\left(n^{2k-4-i}\right)
\quad{\rm for}\quad 2\le i \le k-3. \notag
\end{align}

Now let $S\subset V$ be a set of size $k-1$.
By the definition of $\mathcal{H}_{k}$, $d_{\mathcal{H}_{k}}(S) \le n^{2/(k-1)}/(\log n)^{\beta}$.
Therefore,
\begin{align}
C_{\mathcal{H}_k}(2,k-1)
\le |\mathcal{H}_k| \times \binom{k}{k-1} \times n^{\frac{k-3}{k-1}}/(\log n)^{\beta}
= O\left(n^{k-2+\frac{k-3}{k-1}}/(\log n)^{\beta}\right). \notag
\end{align}
\end{proof}

Since
\begin{align}
1+\left(k-4+\frac{2}{k-1}\right)\frac{2k-1-i}{k-2} - \left(2k-4-i\right)
= \frac{2(i-1)}{k-1}>\epsilon, \notag
\end{align}
by Claim~\ref{CLAIM:H-k},
\begin{align}
C_{\mathcal{H}_k}(2,i)
= O\left(n^{2k-4-i}\right)
= o\left(n\left({D}/{\log n}\right)^{\frac{2k-i-1}{k-1-1}}\right) \quad{\rm for}\quad 2 \le i \le k-3. \notag
\end{align}
Moreover, $C_{\mathcal{H}_k}(2,k-2) = 0 \ll n\left({D}/{\log n}\right)^{\frac{2k-(k-2)-1}{k-1-1}}$,
and
\begin{align}
C_{\mathcal{H}_k}(2,k-1) =
 O\left(\frac{n^{k-2+\frac{k-3}{k-1}}}{(\log n)^{\beta}}\right)
\ll
\frac{n^{k-2+\frac{k-3}{k-1}}}{(\log n)^{(1-\beta)\frac{k}{k-2}}}=
n\left(\frac{D}{\log n}\right)^{\frac{2k-(k-1)-1}{k-1-1}}, \notag
\end{align}
where the inequality follows from the assumption that $\beta> \frac{k}{2(k-1)}$.
Therefore, $\mathcal{H}_k$ satisfies condition $(d)$ in Theorem~\ref{THM:common-indep-set}.

So, by Theorem~\ref{THM:common-indep-set}, there exists a set $I\subset V$ of size
$\Omega\left(\omega\cdot {n}/{n^{\frac{k-3}{k-1}}}\right) = \Omega\left(n^{{2}/{(k-1)}}\omega\right)$
such that $I$ is independent in both $\mathcal{H}_{k-1}$ and $\mathcal{H}_k$.
Here
\begin{align}
\omega
= \left( \log\left( \left( (\log n)/{D} \right)^{ \frac{k_2-1}{k_1-1} } d \right) \right)^{1/(k_2-1)}
& = \left( \log \left( \log n \right)^{ (1-\beta) \frac{k-1}{k-2} } \right)^{1/(k-1)} \notag\\
& = \Omega\left( \left(\log \log n\right)^{1/(k-1)} \right). \notag
\end{align}
\end{proof}

\section{Proof of Theorem~\ref{THM:indep-number-ell-small}}\label{SEC:omit-system-ell-small}

\subsection{Lower bound}
We prove the lower bound in Theorem~\ref{THM:indep-number-ell-small} in this section.
The proof idea is similar to that used in the proof of Theorem~\ref{THM:indep-ell-k-2},
that is, we decompose an $(n,k,\ell)$-omitting system into many different hypergraphs so that each hypergraph
contains the information of a certain subset of edges in the original hypergraph.
Then we use a probabilistic argument to show that there exists a large common independent set of these hypergraphs.

Recall that an $n$-vertex $k$-graph $\mathcal{H}$ is an $(n,k,\ell,\lambda)$-omitting system iff it is
$S_{\lambda+1}(\ell)$-free.
While Theorem~\ref{THM:indep-number-ell-small} as stated provides a lower bound on the independence number
of $(n,k,\ell)$-omitting systems, the result holds in the more general setting of $(n,k,\ell,\lambda)$-omitting systems.
We present the proof in this more general setting.

Let $k\ge k_0 > \ell \ge 1$, $\lambda \ge 2$, and $\mathcal{H}$ be an $S_{\lambda}(\ell)$-free $k$-graph.
We say $\mathcal{H}$ is {\em $k_0$-indecomposable} if
\begin{itemize}
\item $k = k_0$, or
\item $k > k_0$ and $\mathcal{H}$ is $\{S_{\lambda_1}(k-1),\ldots,S_{\lambda_{k-k_0}}(k_0)\}$-free,
        where $\lambda_{i} = \left(k \lambda\right)^{2^{i-1}}$ for $i\in[k-k_0]$.
\end{itemize}
Otherwise, we say $\mathcal{H}$ is {\em $k_0$-decomposable}.

Call a family $\mathcal{F}$ of hypergraphs $k_0$-indecomposable if every member in it is $k_0$-indecomposable.
Otherwise, we say $\mathcal{F}$ is $k_0$-decomposable.

\textbf{The decomposition algorithm.}\\
\textbf{Input:} An $S_{\lambda}(\ell)$-free $k$-graph $\mathcal{H}$ and a threshold $k_0$ with $k\ge k_0> \ell$.\\
\textbf{Output:} A family $\mathcal{F}$ of $S_{\lambda}(\ell)$-free $k_0$-indecomposable hypergraphs.\\
\textbf{Operation:} We start with the family $\mathcal{F} = \{\mathcal{H}\}$.
If $\mathcal{F}$ is $k_0$-indecomposable, then we terminate this algorithm.
Otherwise, let $\mathcal{G} \in \mathcal{F}$ be a $k_0$-decomposable hypergraph
and let $k'$ denote the size of each edge in $\mathcal{G}$.
Let $i_0 \in \{1,\ldots,k'-k_0\}$ be the smallest integer such that $\mathcal{G}$ contains a copy of $S_{\lambda_{i_0}}(k'-i_0)$,
where $\lambda_{i_0} = \left(k \lambda\right)^{2^{i_0-1}}$.
Define
\begin{align}
\mathcal{G}_{k'-i_0} = \left\{A \in \binom{V(\mathcal{H})}{k'-i_0}: d_{\mathcal{G}}(A) \ge \lambda_{i_0} \right\}
\quad{\rm and}\quad
\mathcal{G}_{k'} = \left\{B \in \mathcal{G}: \binom{B}{k'-i_0}\cap \mathcal{G}_{k'-i_0} = \emptyset \right\}. \notag
\end{align}
Update $\mathcal{F}$ by removing $\mathcal{G}$ and adding $\mathcal{G}_{k'-i_0}$ and $\mathcal{G}_{k'}$.
Repeat this operation until $\mathcal{F}$ is $k_0$-indecomposable.

We need the following lemmas to show that the algorithm defined above always terminates. Write $\nu(\mathcal{H})$ for the size of a maximum matching in $\mathcal{H}$.

\begin{lemma}\label{LEMMA:S-free-hypergraph-has-large-matching-number}
Let $\mathcal{H}$ be an $\left\{S_{\lambda_{1}}(k-1),\ldots, S_{\lambda_{k-1}}(1) \right\}$-free $k$-graph with $m$ edges.
Then  $$\nu(\mathcal{H})\ge \frac{{m}}{\prod_{i=1}^{k-1}(i+1)\lambda_i}.$$
\end{lemma}
\begin{proof}[Proof of Lemma~\ref{LEMMA:S-free-hypergraph-has-large-matching-number}]
For $j\in [k-1]$ let $\Lambda_j = \prod_{i=1}^{j}(i+1)\lambda_i$.
We prove this lemma by induction on $k$.
Suppose that $k = 2$.
Since $\mathcal{H}$ is $S_{\lambda_{1}}(1)$-free,
$d_{\mathcal{H}}(v) \le \lambda_1-1$ for all $v \in V(\mathcal{H})$.
Therefore, by greedily choosing an edge $e$ and removing all edges that have nonempty intersection with $e$,
we obtain at least $m/(2\lambda_1)$ pairwise disjoint edges in $\mathcal{H}$.

Now suppose that $k \ge 3$.
We claim that $d_{\mathcal{H}}(v) \le (\lambda_{k-1}-1)\Lambda_{k-2}$ for all $v \in V(\mathcal{H})$.
Indeed, suppose to the contrary that there exists $v_0 \in V(\mathcal{H})$
with $d_{\mathcal{H}}(v_0) \ge (\lambda_{k-1}-1)\Lambda_{k-2} +1$.
Since $\mathcal{H}$ is $\left\{S_{\lambda_{1}}(k-1),\ldots, S_{\lambda_{k-2}}(2)\right\}$-free,
the link $L_{\mathcal{H}}(v_0)$ is $\left\{S_{\lambda_{1}}(k-2),\ldots, S_{\lambda_{k-2}}(1)\right\}$-free.
By the induction hypothesis,
$$\nu(L_{\mathcal{H}}(v_0)) \ge \frac{(\lambda_{k-1}-1)\Lambda_{k-2} +1}{\Lambda_{k-2}} > \lambda_{k-1}-1,$$
but this contradicts the assumption that $\mathcal{H}$ is $S_{\lambda_{k-1}}(1)$-free.
Therefore, $d_{\mathcal{H}}(v) \le (\lambda_{k-1}-1)\Lambda_{k-2}$ for all $v \in V(\mathcal{H})$.
Then,
similar to the case of $k=2$, by greedily choosing an edge $e$ and removing all edges that have nonempty intersection with $e$,
we obtain
$$\nu(\mathcal{H}) \ge \frac{m}{k(\lambda_{k-1}-1)\Lambda_{k-2}+1} > \frac{m}{\Lambda_{k-1}}$$
completing the proof.
\end{proof}

Let $\mathcal{H}$ be an $S_{\lambda}(\ell)$-free $k$-graph.
Define
\begin{align}
\mathcal{H}_{k-1} = \left\{A\in \binom{V(\mathcal{H})}{k-1}: d_{\mathcal{H}}(A) \ge k \lambda \right\}. \notag
\end{align}
If $\mathcal{H}$ is $\left\{S_{\lambda'_{1}}(k-1), \ldots, S_{\lambda'_{k-k'-1}}(k'+1),S_{\lambda}(\ell)\right\}$-free
for some $\ell < k'\le k-2$,
then also define
\begin{align}
\mathcal{H}_{k'} = \left\{A \in \binom{V}{k'}: d_{\mathcal{H}}(A) \ge
k \lambda \prod_{i=1}^{k-k'-1}(i+1)\lambda'_i \right\}. \notag
\end{align}

\begin{lemma}\label{LEMMA:decomposed-hypergraphs-S-free}
The hypergraphs $\mathcal{H}_{k'}$ and $\mathcal{H}_{k-1}$ defined above are $S_{\lambda}(\ell)$-free.
\end{lemma}
\begin{proof}[Proof of Lemma~\ref{LEMMA:decomposed-hypergraphs-S-free}]
We may only prove that $\mathcal{H}_{k'}$ is $S_{\lambda}(\ell)$-free,
since the proof for $\mathcal{H}_{k-1}$ is basically the same.
Suppose to the contrary that
there exists $\{A_1,\ldots,A_{\lambda}\} \subset \mathcal{H}_{k'}$ forming a copy of $S_{\lambda}(\ell)$.
Since $\mathcal{H}$ is $\{S_{\lambda'_{1}}(k-1), \ldots, S_{\lambda'_{k-k'-1}}(k'+1)\}$-free,
the link $L_{\mathcal{H}}(A_i)$ is $\{S_{\lambda'_{1}}(k-k'-1), \ldots, S_{\lambda'_{k-k'-1}}(1)\}$-free
for $i\in[\lambda]$.
Let $\Lambda' = \prod_{i=1}^{k-k'-1}(i+1)\lambda'_i$.
It follows from the definition of $\mathcal{H}_{k'}$ that
$|L_{\mathcal{H}}(A_i)| \ge k \lambda \Lambda'$ for $i\in[\lambda]$.
So, by Lemma~\ref{LEMMA:S-free-hypergraph-has-large-matching-number},
there are at least $k \lambda \Lambda'/\Lambda'\ge k \lambda$
pairwise disjoint edges in $L_{\mathcal{H}}(A_i)$ for $i\in[\lambda]$.
Therefore, there exist $\lambda$ pairwise disjoint $(k-k')$-sets $B_1,\ldots,B_{\lambda}$
such that $B_i \subset V\setminus\left(\bigcup_{i=1}^{\lambda}A_i\right)$
and $E_i= A_i \cup B_i \in \mathcal{H}$ for $i \in[\lambda]$.
It is clear that $\{E_1,\ldots,E_{\lambda}\}$ is a copy of $S_{\lambda}(\ell)$ in $\mathcal{H}$, a contradiction.
\end{proof}

Recall that in the decomposition algorithm we defined
\begin{align}
\mathcal{G}_{k'-i_0} = \left\{A \in \binom{V(\mathcal{H})}{k'-i_0}: d_{\mathcal{G}}(A) \ge \lambda_{i_0} \right\},
\quad{\rm and}\quad
\mathcal{G}_{k'} = \left\{B \in \mathcal{G}: \binom{B}{k'-i_0}\cap \mathcal{G}_{k'-i_0} = \emptyset \right\},\notag
\end{align}
where $i_0 \in \{1,\ldots,k'-k_0\}$ is the smallest integer such that $\mathcal{G}$ contains a copy of $S_{\lambda_{i_0}}(k'-i_0)$
and $\lambda_{i_0} = \left(k \lambda\right)^{2^{i_0-1}}$.
It is clear from the definition that $\mathcal{G}_{k'}$ is $S_{\lambda_{i_0}}(k'-i_0)$-free.
On the other hand,
Lemma~\ref{LEMMA:decomposed-hypergraphs-S-free} implies that $\mathcal{G}_{k'-i_0}$ is $S_{\lambda}(\ell)$-free.
Therefore, the new hypergraphs $\mathcal{G}_{k'-i_0}$ and $\mathcal{G}_{k'}$ we added into $\mathcal{F}$
either have a smaller edge size (the case $\mathcal{G}_{k'-i_0}$) or forbid one more hypergraph (the case $\mathcal{G}_{k'}$).
So the algorithm must terminate after finite many steps, and
it is easy to see that the outputted family $\mathcal{F}$ has size at most $2^{k-k_0}$.

The following lemma shows that in order to find a large independent set in $\mathcal{H}$
it suffices to find a large common independent set of all hypergraphs in $\mathcal{F}$.

\begin{lemma}\label{LEMMA:common-indep-set-of-F-is-indep-in-H}
Let $\mathcal{H}$ be an $S_{\lambda}(\ell)$-free $k$-graph and $\mathcal{F}$ be the outputted family
after applying the decomposition algorithm to $\mathcal{H}$.
Then
\begin{align}
\alpha(\mathcal{H}) \ge \alpha\left(\bigcup_{\mathcal{G}\in \mathcal{F}}\mathcal{G}\right). \notag
\end{align}
\end{lemma}
\begin{proof}[Proof of Lemma~\ref{LEMMA:common-indep-set-of-F-is-indep-in-H}]
Suppose that $\mathcal{F} = \{\mathcal{H}_1,\ldots, \mathcal{H}_m\}$
and $I\subset V(\mathcal{H})$ is independent in $\mathcal{H}_i$ for $i\in[m]$.
It is clear from the definition that for every $E\in \mathcal{H}$
there is a subset $E'\subset E$ such that $E'\in \mathcal{H}_i$ for some $i\in[m]$.
Since $I$ is independent $\mathcal{H}_i$, $E' \not\subset I$ and it follows that $E\not\subset I$.
Therefore, $I$ is independent in $\mathcal{H}$.
\end{proof}

We also need the following lemma which gives a upper bound for the size of an indecomposable hypergraph.

\begin{theorem}[Deza-Erd\H{o}s-Frankl \cite{DEF78}]\label{THM:Deza-Erdos-Frankl-L-system}
Let $r\ge 1$, $t\ge 2$ be integers and $L = \{\ell_1,\ldots,\ell_{r}\}$ be a set of integers with
$0 \le \ell_1<\cdots<\ell_r<k$.
If an $n$-vertex $k$-graph $\mathcal{H}$ is $S_{t}(\ell)$-free for every $\ell\in [k]\setminus L$,
then $|\mathcal{H}| = O(n^{r-1})$ unless
$(\ell_2-\ell_1)\mid \cdots \mid (\ell_r-\ell_{r-1}) \mid (k-\ell_r)$.
\end{theorem}

\begin{lemma}\label{LEMMA:size-indecomposable-hypergraphs}
Let $\ell \ge 1$, $k\ge k_0 >\ell \ge 1$, $\lambda \ge 2$ be integers, $k>2\ell+1$, $k_0\ge \ell+3$,
and $\mathcal{H}$ be a $S_{\lambda}(\ell)$-free $k_0$-indecomposable $k$-graph with $n$ vertices.
Then there exists a constant $C_{k,\ell,\lambda}$ such that
$|\mathcal{H}| \le C_{k,\ell,\lambda}n^{\min\{k_0-2,k-\ell-1\}}$.
\end{lemma}
\begin{proof}[Proof of Lemma~\ref{LEMMA:size-indecomposable-hypergraphs}]
Since $\mathcal{H}$ is $S_{\lambda}(\ell)$-free and $k>2\ell+1$, by the results in \cite{FF85},
$|\mathcal{H}| = O\left(n^{k-\ell-1}\right)$.
On the other hand, since $\mathcal{H}$ is $\{S_{\lambda_1}(k-1),\ldots,S_{\lambda_{k-k_0}}(k_0),S_{\lambda}(\ell)\}$-free,
applying Theorem~\ref{THM:Deza-Erdos-Frankl-L-system} to $\mathcal{H}$ with $t = \max\{\lambda_1,\ldots,\lambda_{k-k_0},\lambda\}$
and $L = \{0,1,\ldots,\ell-1,\ell+1,\ldots,k_0-1\}$
we obtain $|\mathcal{H}| = O\left(n^{k_0-2}\right)$.
\end{proof}

Now we are ready to prove the lower bound in Theorem~\ref{THM:indep-number-ell-small}.

\begin{proof}[Proof of the lower bound in Theorem~\ref{THM:indep-number-ell-small}]
We may assume that $k> 3\ell$ since otherwise by $(\ref{equ:lower-bound-g-Spencer})$ we are done.
Let $\mathcal{H}$ be an $S_{\lambda}(\ell)$-free $k$-graph on $n$ vertices and $V = V(\mathcal{H})$.
Apply the decomposition algorithm to $\mathcal{H}$ with the threshold $k_0=2\ell+1$,
and let $\mathcal{F}$ denote the outputted family.
Suppose that $\mathcal{F} = \{\mathcal{H}_{1},\ldots,\mathcal{H}_{m}\}$ for some integer $m$.
For $i\in[m]$ let $k_i$ denote the size of each edge in $\mathcal{H}_{i}$ and note from the definition of the algorithm
that $2\ell+1 \le k_i \le k$.
Let $C = \max\{C_{k_i,\ell,\lambda}: 2\ell+1\le k_i \le k\}$, where $C_{k_i,\ell,\lambda}$ is the constant given by
Lemma~\ref{LEMMA:size-indecomposable-hypergraphs}.
Choose a set $I \subset V$ such that every vertex is included in $I$ independently with probability
$p = \delta n^{-\frac{2\ell-2}{3\ell-1}}$, where $\delta > 0$ is a small constant
that satisfies $C m \delta^{2\ell} \le 1/4$.
Then by Lemma~\ref{LEMMA:size-indecomposable-hypergraphs},
\begin{align}
\mathbb{E}\left[ |I|-\sum_{i=1}^{m}\left|\mathcal{H}_{i}[I]\right| \right]
& = \mathbb{E}[|I|] - \sum_{i=1}^{m}\mathbb{E}[|\mathcal{H}_{i}[I]|] \notag\\
& \ge pn- \sum_{i=1}^{m} C p^{k_i}  n^{\min\{2\ell-1,k_i-\ell-1\}} \notag\\
& = pn- C \left(\sum_{i\in[m]: k_i \ge 3\ell}  p^{k_i}  n^{2\ell-1}
      + \sum_{i\in[m]: k_i \le 3\ell-1}  p^{k_i}  n^{k_i-\ell-1}\right)\notag\\
& \ge \delta n^{\frac{\ell+1}{3\ell-1}} - C m \delta^{3\ell} n^{\frac{\ell+1}{3\ell-1}} - C m \delta^{2\ell+1} n^{\frac{\ell+1}{3\ell-1}}
 \ge \delta n^{\frac{\ell+1}{3\ell-1}}/{2}. \notag
\end{align}
Therefore, there exists a set $I$ of size $\Omega\left(n^{\frac{\ell+1}{3\ell-1}}\right)$
such that $\mathcal{H}_i[I] = \emptyset$ for $i\in[m]$,
and it follow from Lemma~\ref{LEMMA:common-indep-set-of-F-is-indep-in-H} that
$\alpha(\mathcal{H}) \ge |I| =\Omega\left(n^{\frac{\ell+1}{3\ell-1}}\right)$.
\end{proof}

{\bf Remark.}
The lower bound $n^{\frac{\ell+1}{3\ell-1}}$ can be improved by optimizing the choice of $k_0$.
Indeed, suppose that $\ell$ is sufficiently large. Repeating the argument above we obtain
\begin{align}
\mathbb{E}\left[ |I|-\sum_{i=1}^{m}\left|\mathcal{H}_{i}[I]\right| \right]
& = \mathbb{E}[|I|] - \sum_{i=1}^{m}\mathbb{E}[|\mathcal{H}_{i}[I]|] \notag\\
& = pn - \left(\sum_{k_i\ge s}\mathbb{E}[|\mathcal{H}_{i}[I]|]
                + \sum_{2\ell+1\le k_i\le s}\mathbb{E}[|\mathcal{H}_{i}[I]|]
                + \sum_{k_0<k_i\le 2\ell}\mathbb{E}[|\mathcal{H}_{i}[I]|]\right) \notag\\
& \ge pn -  \left(\sum_{k_i\ge s}p^{k_i}n^{k_0-2}
                + \sum_{2\ell+1\le k_i< s}p^{k_i}n^{k_i-\ell-1}
                + \sum_{k_0\le k_i\le 2\ell}p^{k_i}n^{\ell}\right) \notag\\
& \ge pn -  Cm\left(p^{s}n^{k_0-2} + p^{s-1}n^{s-\ell-2}+p^{k_0}n^{\ell}\right), \notag
\end{align}
where $k_0$ and $s$ will be determined later.

Notice that we need $p \ll \min\left\{n^{-\frac{k_0-3}{s-1}}, n^{-\frac{s-\ell-3}{s-2}}, n^{-\frac{\ell-1}{k_0-1}} \right\}$.
So, by letting $\frac{k_0-3}{s-1} = \frac{s-\ell-3}{s-2} = \frac{\ell-1}{k_0-1}$, we obtain
$k_0^2-k_0\ell-\ell^2 \sim 0$ and $s\sim k_0+\ell$.
This implies that we should let
\begin{align}
k_0 =\left(\frac{\sqrt{5}+1}{2}+o_{\ell}(1) \right)\ell,
\quad
s = \left(\frac{\sqrt{5}+3}{2}+o_{\ell}(1) \right)\ell,
\quad{\rm and}\quad
p = \delta n^{-\left(\frac{\sqrt{5}-1}{2}+o_{\ell}(1) \right)}\notag
\end{align}
for some sufficiently small $\delta>0$.
Then we obtain
\begin{align}
\mathbb{E}\left[ |I|-\sum_{i=1}^{m}\left|\mathcal{H}_{i}[I]\right| \right]
\ge pn -  Cm\left(p^{s}n^{k_0-2} + p^{s-1}n^{s-\ell-2}+p^{k_0}n^{\ell}\right)
\ge \delta n^{\left(\frac{3-\sqrt{5}}{2}+o_{\ell}(1) \right)}/2,\notag
\end{align}
which implies that
$\mathcal{H}$ contains an independent set $I$ of size $\Omega\left(n^{\left(\frac{3-\sqrt{5}}{2}+o_{\ell}(1) \right)}\right)$.

Similarly, the lower bound for $g(n,6,2)$ can be improved from $\Omega\left(n^{3/5}\right)$
to $\Omega\left(n^{2/3}\right)$ by letting $k_0=4$.
Indeed, it is easy to see that when applying the decomposition algorithm to an $n$-vertex $S_{\lambda}(2)$-free $6$-graph $\mathcal{H}$
with the threshold $k_0=4$,
the outputted family $\mathcal{F}$  consists of three hypergraphs:
an $S_{\lambda}(2)$-free $4$-indecomposable $6$-graph $\mathcal{H}_1$,
an $S_{\lambda}(2)$-free $4$-indecomposable $5$-graph $\mathcal{H}_2$,
and an $S_{\lambda}(2)$-free $4$-graph $\mathcal{H}_3$.
By Theorem~\ref{THM:FF85} (the stronger version in \cite{FF85}),
$|\mathcal{H}_2| = O\left(n^2\right)$ and $|\mathcal{H}_3| = O\left(n^2\right)$.
By Theorem~\ref{THM:Deza-Erdos-Frankl-L-system}, $\mathcal{H}_1 = O\left(n^2\right)$.
So, it follows from a similar probabilistic argument as above that
$\alpha(\mathcal{H}) \ge \alpha(\mathcal{H}_1\cup \mathcal{H}_2\cup \mathcal{H}_3) = \Omega\left(n^{2/3}\right)$.
\subsection{Pseudorandom bipartite graphs}
Our construction for the upper bound in Theorem~\ref{THM:indep-number-ell-small} is related to
some pseudorandom bipartite graphs, so it will be convenient to
introduce some definitions and results related to pseudorandom bipartite graphs.

For a graph $G$ on $n$ vertices (assuming that $V(G) = [n]$)
the {\em adjacency matrix} $A_{G}$ of $G$ is an $n \times n$ matrix whose $(i,j)$-th entry is
\begin{align}
A_{G}(i,j) =
\begin{cases}
   1, & \mbox{if } \{i,j\}\in E(G), \\
   0, & {\rm otherwise.}
\end{cases} \notag
\end{align}

Denote by $G(V_1,V_2)$ a bipartite graph with two parts $V_1$ and $V_2$,
and that say $G(V_1,V_2)$ is {\em $(d_1,d_2)$-regular} if $d_{G}(v) = d_i$ for all $v \in V_i$ and $i=1,2$.

The following fact about $(d_1,d_2)$-regular bipartite graphs is well-known and easy to prove using the Perron-Frobenius theorem.

\begin{fact}\label{FACT:bipartite-gp-first-eigen}
Let $G = G(V_1,V_2)$ be a $(d_1,d_2)$-regular bipartite graph with $|V_1| = m$ and $|V_2| = n$.
Then $\sqrt{d_1d_2}$ and $-\sqrt{d_1d_2}$ are the largest and smallest eigenvalues of $A_{G}$ respectively,
and the corresponding eigenvectors are
\begin{align}
& \vec{v}^{+}(m,n) = \left(\underbrace{(2m)^{-1/2}, \ldots, (2m)^{-1/2}}_{m \text{ times}},
\underbrace{(2n)^{-1/2}, \ldots,(2n)^{-1/2}}_{n \text{ times}}\right)^{T},
\quad {\rm and} \notag\\
& \vec{v}^{-}(m,n) = \left(\underbrace{(2m)^{-1/2}, \ldots, (2m)^{-1/2}}_{m \text{ times}},
\underbrace{-(2n)^{-1/2}, \ldots,-(2n)^{-1/2}}_{n \text{ times}}\right)^{T}. \notag
\end{align}
\end{fact}

For a bipartite $G = G(V_1,V_2)$ denote by $\lambda(G)$ the second largest eigenvalue of $A_{G}$.
Suppose that $G$ is $(d_1,d_2)$-regular.
Then we say $G$ is {\em pseudorandom} if $\lambda(G) = O\left(\max\{\sqrt{d_1},\sqrt{d_2}\}\right)$.

The well known Expander mixing lemma for regular graphs (e.g. see \cite{AC88})
relates the distribution of edges in a graph to spectral properties of its adjacency matrix.
Similar results for regular bipartite graphs were also obtained by several authors (e.g. see \cite{Hae95,KSV13,KMV13}).
Here we show a similar result for $(d_1,d_2)$-regular bipartite graphs, and our proof is basically the same as the
original proof of the Expander mixing lemma.

\begin{lemma}\label{LEMMA:pseudorandom-two-bipartite-inequality}
Let $G = G(V_1,V_2)$ be a $(d_1,d_2)$-regular bipartite graph with $|V_1| = m$ and $|V_2| = n$.
Then for every $X \subset V_1$ and $Y \subset V_2$,
the number $e(X,Y)$ of edges between $X$ and $Y$ satisfies
\begin{align}
\left| e(X,Y) - \frac{d_1}{n}|X||Y| \right| \le \lambda(G) \sqrt{|X||Y|}. \notag
\end{align}
\end{lemma}
\begin{proof}[Proof of Lemma~\ref{LEMMA:pseudorandom-two-bipartite-inequality}]
Let $\lambda_1 \ge \cdots \ge \lambda_{m+n}$ be the eigenvalues of $A_G$
and let $\vec{v}_1, \ldots, \vec{v}_{m+n}$ be an orthonormal basis of eigenvectors of $A_G$,
where $\vec{v}_i$ is the eigenvector corresponding to $\lambda_i$ for $i\in [m+n]$.
Since $G$ is $(d_1,d_2)$-regular, by Fact~\ref{FACT:bipartite-gp-first-eigen}, $\lambda_1 = -\lambda_{m+n} = \sqrt{d_1d_2}$,
and we may assume that $\vec{v}_1 = \vec{v}^{+}(m,n)$ and $\vec{v}_{m+n} = \vec{v}^{-}(m,n)$.
Let $\chi_X$ and $\chi_Y$ denote the characteristic vector of $X$ and $Y$, and write
$\chi_X = \sum_{i=1}^{m+n}s_i\vec{v}_i$ and $\chi_Y = \sum_{i=1}^{m+n}t_i\vec{v}_i$.
Then
\begin{align}
e(X,Y)
= \langle A_{G} \chi_{X}, \chi_{Y}\rangle
& = \lambda_{1}s_1t_1 + \lambda_{m+n}s_{m+n}t_{m+n} + \sum_{i=2}^{m+n-1}\lambda_{i}s_it_i \notag\\
& = 2 \sqrt{d_1d_2} |X| \sqrt{1/2n} |Y| \sqrt{1/2m}
     + \sum_{i=2}^{m+n-1}\lambda_{i}s_it_i \notag\\
& = \frac{d_1}{n}|X||Y| + \sum_{i=2}^{m+n-1}\lambda_{i}s_it_i. \notag
\end{align}
Here we used the fact that $md_1 = nd_2 = |G|$.

Since
\begin{align}
\left| \sum_{i=2}^{m+n-1}\lambda_{i}s_it_i \right|
\le \lambda(G) \left(\sum_{i=2}^{m+n-1}s_i^2\right)^{1/2}\left(\sum_{i=2}^{m+n-1}t_i^2\right)^{1/2}
\le \lambda(G)\sqrt{|X||Y|}, \notag
\end{align}
we obtain
\begin{align}
\left| e(X,Y) - \frac{d_1}{n}|X||Y| \right| \le \lambda(G)\sqrt{|X||Y|}. \notag
\end{align}
\end{proof}

The {\em Zarankiewicz number} $z(m,n,s,t)$ is the maximum number of edges in a bipartite graph $G(V_1,V_2)$
with $|V_1| = m$, $|V_2| = n$ such that $G$ contains no complete bipartite graph with $s$ vertices in $V_1$ and $t$ vertices in $V_2$.

Our construction of $(n,k,\ell)$-systems is related to the lower bound (construction) for $z(m,n,s,t)$.
More specifically, it is related to a construction defined by
Alon, Mellinger, Mubayi and Verstra\"{e}te in \cite{AMMV12},
which was used to show that $z(n^{\ell/2},n,2,\ell) = \Omega\left(n^{(\ell+1)/2}\right)$.

Let $q$ be a prime power and $\mathbb{F} = GF(q)$ be the finite field of size $q$.
Denote by $\mathbb{F}[X]$ the collection of all polynomials over $\mathbb{F}$.
The graph $G(q^{\ell},q^{2},2,\ell)$ is a bipartite graph with two parts $V_1$ and $V_2$,
where
\begin{align}
V_1 = \left\{P(x): P(x)\in \mathbb{F}[X], \deg(P(x)) \le \ell-1 \right\},
\quad{\rm and}\quad
V_2 = \mathbb{F} \times \mathbb{F}, \notag
\end{align}
and for every $P(x) \in V_1$ and every $(x,y) \in V_2$, the pair $\{P(x),(x,y)\}$ is an edge in $G(q^{\ell},q^{2},2,\ell)$ iff $y = P(x)$.

It is clear that $G(q^{\ell},q^{2},2,\ell)$ does not contain a complete bipartite graph with two vertices in $V_1$ and $\ell$ vertices in $V_2$
since two distinct polynomials of degree at most $\ell-1$ over $\mathbb{F}$ can have the same value in at most $\ell-1$ points.
It is also easy to see that $G(q^{\ell},q^{2},2,\ell)$ is $(q, q^{\ell-1})$-regular.

The proof of the following result concerning the eigenvalues of $G(q^{\ell},q^{2},2,\ell)$ can be found in \cite{EV13}.

\begin{lemma}[\cite{EV13}]\label{LEMMA:Zarankiewicz-construction-eigenvalues}
The eigenvalues of the adjacency matrix of $G(q^{\ell},q^{2},2,\ell)$ are
\begin{align}
q^{\ell/2}, \underbrace{q^{(\ell-1)/2}, \ldots, q^{(\ell-1)/2}}_{q^2-q \text{ times}},
0, \ldots, 0, \underbrace{-q^{(\ell-1)/2}, \ldots, -q^{(\ell-1)/2}}_{q^2-q \text{ times}},-q^{\ell/2}. \notag
\end{align}
In particular, $G(q^{\ell},q^{2},2,\ell)$ is pseudorandom.
\end{lemma}

\subsection{Upper bound}
In this section we prove the existence of $(n,k,\ell)$-systems with independence number
$O\left(n^{\frac{\ell+1}{2\ell}}(\log n)^{\frac{1}{\ell}}\right)$.
Our construction is obtained from a random subgraph of
the bipartite graph $G(q^{\ell},q^{2},2,\ell)$ defined in the last section,
and the method we used here is similar to that used in \cite{KMV13,EV13}.

First let us summarize the constructions used in \cite{KMV13} and \cite{EV13} into a more general form.

Since we cannot ensure the random subgraph chosen from $G(q^{\ell},q^{2},2,\ell)$
is exactly $(d_1,d_2)$-regular for some $d_1,d_2\in \mathbb{N}$,
it will be useful to consider the following more general setting.

Let $C,d_1,d_2 \ge 1$ be real numbers.
A hypergraph $\mathcal{H}$ is
\begin{itemize}
\item[(a)] {\em $(C,d_1)$-uniform} if $d_1/C \le |E| \le C d_1$ for all $E\in \mathcal{H}$, and
\item[(b)] {\em $(C,d_2)$-regular} if $d_2/C \le d_{\mathcal{H}}(v) \le Cd_2$ for all $v \in V(\mathcal{H})$.
\end{itemize}

The {\em edge density} of a $k$-graph $\mathcal{H}$ with $n$ vertices is $\rho(\mathcal{H}) = |\mathcal{H}|/\binom{n}{k}$.
The {\em bipartite incidence graph} $G_{\mathcal{H}}$ of $\mathcal{H}$ is a bipartite graph
with two parts $V_1 = E(\mathcal{H})$ and $V_2 = V(\mathcal{H})$,
and for every $E\in E(\mathcal{H})$ and $v\in V(\mathcal{H})$
the pair $\{E,v\}$ is an edge in $G_{\mathcal{H}}$ iff $v\in E$.
Denote by $A_{\mathcal{H}}$ the adjacency matrix of $G_{\mathcal{H}}$.

Let $n = |V(\mathcal{H})|$, $m = |\mathcal{H}|$ and
labelling the edges in $\mathcal{H}$ with $E_1,\ldots,E_{m}$
We say a family $\mathcal{F}$ of hypergraphs
{\em fits} $\mathcal{H}$ if $\mathcal{F} = \{\mathcal{G}_i: 1\le i \le m\}$ and
$\mathcal{G}_i$ is a hypergraph with $|V(\mathcal{G}_i)| = |E_i|$ for $i\in[m]$.

Given a hypergraph $\mathcal{H}$ and a family $\mathcal{F}$ that fits $\mathcal{H}$
we let $\mathcal{H}(\mathcal{F})$ be the hypergraph obtained from $\mathcal{H}$ by taking
independently for every $i\in [m]$ a bijection $\psi_{i}: E_i \to V(\mathcal{G}_i)$ and letting
a set $S\subset E_i$ to be an edge in $\mathcal{H}(\mathcal{F})$ if $\psi_{i}(S) \in \mathcal{G}_i$.

Let $\tau \ge 1$ be an integer and denote by $B_{\tau}(\mathcal{G})$ the collection of
$\tau$-subsets of $V(\mathcal{G})$ that are not independent in $\mathcal{G}$.
Let $b_{\tau}(\mathcal{G}) = |B_{\tau}(\mathcal{G})|$ and
$p_{\tau}(\mathcal{G}) = b_{\tau}(\mathcal{G})/\binom{v(\mathcal{G})}{\tau}$.
In other words, $p_{\tau}(\mathcal{G})$ is the probability that a random $\tau$-subset of $V(\mathcal{G})$
is not independent in $\mathcal{G}$.
For a family $\mathcal{F}$ of hypergraphs define
\begin{align}
p_{\tau}(\mathcal{F}) = \min\left\{p_{\tau}(\mathcal{G}): \mathcal{G}\in \mathcal{F}\right\}. \notag
\end{align}

We extend the definition of $C_{\mathcal{G}}(2,j)$ in Section~\ref{SEC:omit-system-ell-big} by letting
$C_{\mathcal{G}}(2,j)$ denote the number of pairs $\{E,E'\}$ in a $k$-graph $\mathcal{G}$ with $|E\cap E'| = j$
for all $0\le j \le k-1$.

The following lemma gives an upper bound for the independence number of $\mathcal{H}(\mathcal{F})$.

\begin{lemma}\label{LEMMA:indep-algebraic-random-constrction}
Let $C,d_1,d_2 \ge 1$ be real numbers and $k\ge 2$ be an integer.
Suppose that $\mathcal{H}$ is a hypergraph with $n$ vertices, $m$ edges, and is $(C,d_1)$-uniform, $(C,d_2)$-regular.
Let $\mathcal{F} = \{\mathcal{G}_i: i\in[m]\}$ be a family of $k$-graphs that fits $\mathcal{H}$.
Suppose there exists $\lambda\ge 0$ such that the bipartite graph $G_{\mathcal{H}}$ satisfies
\begin{align}\label{equ:pseudornadom-bipartite-graph}
\left| e_{G_{\mathcal{H}}}(X,Y) - \frac{d_1}{n}|X||Y| \right| \le \lambda \sqrt{|X||Y|}
\end{align}
for all $X \subset V(\mathcal{H})$ and $Y \subset E(\mathcal{H})$.
Then, w.h.p. $\alpha\left(H(\mathcal{F})\right)\le 2\tau n/d_1$,
if $\tau$ satisfies
\begin{align}\label{equ:tau-conditions}
\frac{p_{\tau}(\mathcal{F})}{\tau}
\ge \frac{8C^2\log n}{d_2} \quad{\rm and}\quad \tau \ge \frac{8C^2\lambda^2}{d_2}.
\end{align}
\end{lemma}
\begin{proof}[Proof of Lemma \ref{LEMMA:indep-algebraic-random-constrction}]
Let $\tau$ be a real number that satisfies $(\ref{equ:tau-conditions})$,
$V = V(\mathcal{H})$, and $I \subset V$ be a set of size $\lceil 2\tau n/d_1 \rceil$
(we may assume that $2\tau n/d_1 \in \mathbb{N}$ to keep the calculations simple).

Let $m = |\mathcal{H}|$ and label the edges in $\mathcal{H}$ by $\{E_1,\ldots,E_{m}\}$.
Let $m_i = |E_i|$ for $i\in[m]$.
Since $\mathcal{H}$ is $(C,d_1)$-uniform and $(C,d_2)$-regular, we obtain
$d_1|\mathcal{H}|/C \le \sum_{v\in V}d_{\mathcal{H}}(v) \le Cd_1|\mathcal{H}|$,
and
\begin{align}\label{equ:vertex-edge-inequality}
md_1/C^2 \le nd_2 \le C^2md_1.
\end{align}

Define
\begin{align}
\mathcal{E}_1  = \left\{E\in \mathcal{H}:  |E\cap I| < \tau \right\},
\quad{\rm and}\quad
\mathcal{E}_2  = \left\{E\in \mathcal{H}:  |E\cap I| >3 \tau \right\}. \notag
\end{align}

\begin{claim}\label{CLAIM:size-E1-E2}
$|\mathcal{E}_i| \le {2C^2\lambda^2 m}/{d_2\tau} \le m/4$ for $i = 1,2$.
\end{claim}
\begin{proof}[Proof of Claim \ref{CLAIM:size-E1-E2}]
It follows from $(\ref{equ:pseudornadom-bipartite-graph})$ that
\begin{align}
\sum_{E\in \mathcal{E}_1}|E\cap I|
= e_{G_{\mathcal{H}}}(I,\mathcal{E}_1)
\ge d_1|I||\mathcal{E}_1|/n - \lambda\left(|I||\mathcal{E}_1|\right)^{1/2}, \notag
\end{align}
and by definition, $\sum_{E\in \mathcal{E}_1}|E\cap I| < \tau|\mathcal{E}_1|$.
Therefore,
\begin{align}
\tau|\mathcal{E}_1| > d_1|I||\mathcal{E}_1|/n - \lambda\left(|I||\mathcal{E}_1|\right)^{1/2}. \notag
\end{align}
Since $|I| = 2\tau n/d_1$, we obtain
\begin{align}
|\mathcal{E}_1|
< \left(\frac{\lambda|I|^{1/2}}{d_1|I|/n-\tau}\right)^2
= \left(\frac{\lambda|I|^{1/2}}{d_1|I|/2n}\right)^2
= \frac{2\lambda^2 n}{\tau d_1},\notag
\end{align}
which together with $(\ref{equ:vertex-edge-inequality})$ implies
$|\mathcal{E}_1| < {2C^2\lambda^2 m}/{d_2\tau}$.
Notice that $(\ref{equ:tau-conditions})$ implies that $C^2\lambda^2 /d_2\tau \le 1/8$,
so $|\mathcal{E}_1| < m/4$.

Now consider $\mathcal{E}_2$.
Similarly, By $(\ref{equ:pseudornadom-bipartite-graph})$,
\begin{align}
\sum_{E\in \mathcal{E}_2}|E\cap I|
= e_{G_{\mathcal{H}}}(I,\mathcal{E}_2)
\le d_1|I||\mathcal{E}_1|/n + \lambda\left(|I||\mathcal{E}_1|\right)^{1/2}, \notag
\end{align}
and by definition, $\sum_{E\in \mathcal{E}_2}|E\cap I| > 3\tau|\mathcal{E}_2|$.
Therefore,
\begin{align}
3\tau|\mathcal{E}_2| < d_1|I||\mathcal{E}_2|/n + \lambda\left(|I||\mathcal{E}_2|\right)^{1/2}, \notag
\end{align}
Since $|I| = 2\tau n/d_1$, we obtain
\begin{align}
|\mathcal{E}_2|
< \left(\frac{\lambda|I|^{1/2}}{3\tau - d_1 |I|/n}\right)^2
= \left(\frac{\lambda|I|^{1/2}}{d_1|I|/2n}\right)^2
= \frac{2\lambda^2 n}{\tau d_1}
\le \frac{2C^2\lambda^2 m}{\tau d_2}
\le \frac{m}{4}.\notag
\end{align}
\end{proof}

For $i\in[m]$ let $I_i = I\cap E_i$.
By Claim \ref{CLAIM:size-E1-E2} the number of set $I_i$ that satisfies $\tau \le |I_i| \le 3\tau$
is at least $m - 2 m/4 = m/2$.
By the definition of $p_{\tau}(\mathcal{F})$, for every $I_i$ that satisfies $\tau \le |I_i| \le 3\tau$ we have
\begin{align}
P\left(I_i \text{ is independent in }\mathcal{H}(\mathcal{F})\right)
\le P\left(\psi_{i}(I_i)\text{ is independent in }\mathcal{G}_i\right)
\le 1- p_{\tau}(\mathcal{F}). \notag
\end{align}
So, the probability that $I$ is independent is at most $\left(1- p_{\tau}(\mathcal{F})\right)^{m/2}$ and hence
the expected number of independent $2n\tau/d_1$-sets in $\mathcal{H}(\mathcal{G})$ is at most
\begin{align}
\left(1- p_{\tau}(\mathcal{F})\right)^{m/2} \binom{n}{2n\tau/d_1}
& < \exp\left(-p_{\tau}(\mathcal{F})\frac{m}{2} + \frac{2\tau n}{d_1}\log\left(\frac{en}{2n\tau/d_1}\right)\right) \notag\\
& < \exp\left(-p_{\tau}(\mathcal{F})\frac{m}{2} + \frac{2C^2\tau m}{d_2}\log n\right) \notag\\
& < \exp\left(-p_{\tau}(\mathcal{F})\frac{m}{4}\right)
\to 0 \quad{\rm as}\quad m \to \infty. \notag
\end{align}
Therefore, $\alpha(\mathcal{H}(\mathcal{F})) \le 2\tau n/d_1$ holds with high probability.
\end{proof}

The following corollary may be a simpler form to use Lemma~\ref{LEMMA:indep-algebraic-random-constrction},
and its proof can be obtained easily using the Inclusion-exclusion principle.

\begin{corollary}\label{CORO:indep-algebraic-random-constrction-few-cycles}
Let $\mathcal{H}$ and $\mathcal{F}$ be the same as in Lemma~\ref{LEMMA:indep-algebraic-random-constrction}.
If there exists $\rho > 0$ such that $\rho(\mathcal{G}_i) \ge \rho$ for $i\in[m]$,
$\lambda < (d_2\tau_1)^{1/2}$ with $\tau_1 = \left({\log n}/{\rho d_2}\right)^{1/(k-1)}$,
and $C_{\mathcal{G}_i}(2,j) \le |\mathcal{G}_i|\left(v(\mathcal{G}_i)/\tau_1\right)^{k-j}$ for $0 \le j \le k-1$ and $i \in [m]$,
then, w.h.p.
$\alpha\left(\mathcal{H}(\mathcal{F})\right) = O\left(\tau_1 n/d_1\right)$.
\end{corollary}

Lemma~\ref{LEMMA:pseudorandom-two-bipartite-inequality} applied to
the bipartite incidence graph $G_{\mathcal{H}}$ of $\mathcal{H}$ gives the following result.

\begin{lemma}[\cite{KMV13}]\label{LEMMA:intersection-inequality-hypergraph}
Let $\mathcal{H}$ be a $d_1$-uniform $d_2$-regular hypergraph on $n$ vertices.
Then for every $V' \subset V(\mathcal{H})$ and $\mathcal{E} \subset E(\mathcal{H})$,
\begin{align}
\left| \sum_{E\in \mathcal{E}}|E\cap V'| -  \frac{d_1}{n}|V'||\mathcal{E}|\right| \le \lambda(G_{\mathcal{H}}) \sqrt{|V'||\mathcal{E}|}.\notag
\end{align}
\end{lemma}

We also need the following Chernoff's inequality (e.g. see Theorem 22.6 in \cite{FK16}).

\begin{theorem}[Chernoff's inequality]
Suppose that $S_n = X_1 + \cdots +X_n$ where $0\le X_i \le 1$ for $i\in[n]$ are independent random variables.
Let $\mu = \mathbb{E}[X_1]+\cdots +\mathbb{E}[X_n]$.
Then for every $0\le t \le \mu$,
\begin{align}
P\left(|S_n-\mu|\ge t\right) \le e^{-\frac{t^2}{3\mu}}. \notag
\end{align}
\end{theorem}

Now we are ready to prove the upper bound in Theorem~\ref{THM:indep-number-ell-small}.

\begin{proof}[Proof of the upper bound in Theorem~\ref{THM:indep-number-ell-small} for the upper bound]
Let $G = G(q^{\ell},q^{2},2,\ell)$ be the bipartite graph on $V_1 \cup V_2$ with $|V_1| =  q^{\ell}$ and $|V_2| = q^2$.
Let $\mathcal{G}$ denote the hypergraph on $q^2$ vertices whose bipartite incident graph is $G$.
Note that $\mathcal{G}$ is a $q^{\ell-1}$-regular $q$-graph,
and by Lemmas~\ref{LEMMA:Zarankiewicz-construction-eigenvalues} and~\ref{LEMMA:intersection-inequality-hypergraph},
\begin{align}\label{equ:G-intersection-inequality}
\left| \sum_{E\in \mathcal{E}}|E\cap V'| -  \frac{1}{q}|V'||\mathcal{E}|\right|
\le q^{(\ell-1)/2} \sqrt{|V'||\mathcal{E}|}
\end{align}
holds for all $V' \subset V(\mathcal{G})$ and $\mathcal{E} \subset \mathcal{G}$.

Let $U \subset V(\mathcal{G})$ be a random set such that every vertex in $V(\mathcal{G})$ is included in $U$
independently with probability $p = q^{-\frac{2}{\ell+1}}$.
Then $\mathbb{E}[|U|] = p q^2 = q^{\frac{2\ell}{\ell+1}}$,
and by the Chernoff inequality,
\begin{align}
P\left(\left||U|-p q^2\right|> p q^2/{2}\right)
< e^{-\frac{\left(p q^2/2\right)^2}{3p q^2}}
= e^{-{p q^2}/{12}} \to 0 \quad{\rm as}\quad q\to \infty. \notag
\end{align}
For every $E\in \mathcal{G}$ we have $\mathbb{E}[|E\cap U|] = p d_1 = q^{\frac{\ell-1}{\ell+1}}$, and by the Chernoff inequality,
\begin{align}
P\left(\left| |E\cap U|-p d_1 \right|> {p d_1}/{2}\right)
< e^{-\frac{(p d_1/2)^2}{3p d_1}}
= e^{-{pd_1}/{12}}. \notag
\end{align}
Let $B$ denote the collection of edges $E\in \mathcal{G}$ such that $\left| |E\cap U|-p d_1 \right|> {p d_1}/{2}$.
Then
\begin{align}
\mathbb{E}[|B|]
\le q^{\ell} e^{-{pd_1}/{12}}
= q^{\ell} e^{-q^{\frac{\ell-1}{\ell+1}}/12}
\to 0 \quad{\rm as}\quad q \to \infty. \notag
\end{align}
Therefore, $w.h.p.$ the set $U$ satisfies that $q^{\frac{2\ell}{\ell+1}}/2 \le |U|\le 3q^{\frac{2\ell}{\ell+1}}/2$ and
$q^{\frac{\ell-1}{\ell+1}}/2 \le |E\cap U| \le 3q^{\frac{\ell-1}{\ell+1}}/2$ for all $E\in \mathcal{G}$.

Fix such a set $U$ that satisfies the conclusion above, and in order to keep the calculations simple
we may assume that $|U| = q^{\frac{2\ell}{\ell+1}}$.
Let $n = |U| = q^{\frac{2\ell}{\ell+1}}$, $m = |\mathcal{G}| = |\mathcal{H}| = q^{\ell}= n^{\frac{\ell+1}{2}}$,
$d_1 = q^{\frac{\ell-1}{\ell+1}} = n^{\frac{\ell-1}{2\ell}}$,
and $d_2 = q^{\ell-1} = n^{\frac{(\ell+1)(\ell-1)}{2\ell}}$.
Let $\mathcal{H}$ be the hypergraph on $U$ with
\begin{align}
\mathcal{H} = \left\{E\cap U: E\in \mathcal{G}\right\}. \notag
\end{align}
Then $\mathcal{H}$ is a $(2,d_1)$-uniform $d_2$-regular hypergraph.
Moreover, $(\ref{equ:G-intersection-inequality})$ also holds for all $V'\subset U$ and $\mathcal{E} \subset \mathcal{G}$.

Label the edges in $\mathcal{G}$ with $\{E_1,\ldots,E_{m}\}$ and let $m_i = |E_i|$ for $i\in[m]$.
Let $\mathcal{F} = \left\{\mathcal{S}_i : i\in [m]\right\}$,
where $\mathcal{S}_i$ is the $k$-graph on $[m_i]$ whose edge set is the collection of all $k$-subsets of $[m_i]$ that contain $[\ell+1]$.
Our construction of the $(n,k,\ell)$-omitting system is simply $\mathcal{H}(k,\ell) = \mathcal{H}(\mathcal{F})$,
and indeed, one can easily check that $|e'\cap e'| \neq \ell$ for all distinct edges $e,e' \in \mathcal{H}(k,\ell)$.

Let $\tau = \lceil 100 \left(\log n\right)^{1/\ell} \rceil$, and to keep the calculations simple we may assume that
$100 \left(\log n\right)^{1/\ell} \in \mathbb{N}$.

\begin{claim}\label{CLAIM:prob-independet}
$p_{\tau}(\mathcal{F}) \ge \left(\frac{\tau}{3d_1/2}\right)^{\ell+1}/2$.
\end{claim}
\begin{proof}[Proof of Claim~\ref{CLAIM:prob-independet}]
Fix $i\in [m]$ and let $I$ be a random $\tau$-subset of $[m_i]$.
It is easy to see that $I$ is not independent in $\mathcal{S}_i$ iff $[\ell+1] \subset I$.
Since
\begin{align}
P\left([\ell+1] \subset I\right)
= \frac{\binom{m_i-\ell-1}{\tau-\ell-1}}{\binom{m_i}{\tau}}
= \frac{\tau\cdots(\tau-\ell)}{m_i \cdots (m_i-\ell)}
\ge (1-o(1))\left(\frac{\tau}{m_i}\right)^{\ell+1}
> \frac{1}{2}\left(\frac{\tau}{3d_i/2}\right)^{\ell+1}, \notag
\end{align}
we obtain
\begin{align}
p_{\tau}(\mathcal{F}) > \frac{1}{2}\left(\frac{\tau}{3d_i/2}\right)^{\ell+1}. \notag
\end{align}
\end{proof}

Observe that $\tau$ satisfies
\begin{align}
\frac{p_{\tau}(\mathcal{F})}{\tau}
> \frac{\left(\frac{\tau}{3d_1/2}\right)^{\ell+1}/2}{\tau}
= \frac{100^{\ell} \log n}{2(3/2)^{\ell+1}d_{1}^{\ell+1}}
= \frac{100^{\ell} }{2(3/2)^{\ell+1}} \frac{\log n}{d_2}
> \frac{32\log n}{d_2} \notag
\end{align}
(here we used the fact that $d_2 = d_1^{\ell+1}$) and
\begin{align}
\tau
= 100 (\log n)^{1/\ell}
> \frac{32\left(q^{(\ell-1)/2}\right)^2}{q^{\ell-1}}. \notag
\end{align}
We may therefore apply Lemma~\ref{LEMMA:indep-algebraic-random-constrction} with $C=2$ to obtain
\begin{align}
\alpha\left(\mathcal{H}(k,\ell)\right) \le {2\tau n}/{d_1} = 200 n^{\frac{\ell+1}{2\ell}} (\log n)^{1/\ell}. \notag
\end{align}
\end{proof}

\section{Independent sets in $(n,k,\ell,\lambda)$-systems}\label{SEC:indepence-number-design}
In this section we prove Theorem~\ref{THM:independence-number-lambda-design}.
Our proof is a direct application of the following theorem due to Duke, Lefmann, and R\"{o}dl \cite{DLR95}.

\begin{theorem}[Duke-Lefmann-R\"{o}dl, \cite{DLR95}]\label{THM:DLR-Uncrowd}
Let $\mathcal{H}$ be a $k$-graph on $n$ vertices satisfying $\Delta(\mathcal{H}) \le t^{k-1}$, where $t \gg k$.
If $C_{\mathcal{H}}(2,j) \le n t^{2k-j-1-\epsilon}$ for $2 \le j \le k-1$ and some constant $\epsilon>0$,
then $\alpha(\mathcal{H}) \ge c(k,\epsilon)\left(\log t\right)^{1/(k-1)} \cdot n/t$.
\end{theorem}

\begin{proof}[Proof of Theorem~\ref{THM:independence-number-lambda-design}]
Fix $\delta>0$, and let $\epsilon>0$ be sufficiently small
such that $\frac{\ell-1}{k-2}-\delta < \frac{(\ell-1)(1-\epsilon)}{k-2+\epsilon}$ holds.
Let $t = \lambda^{\frac{1}{k-1}}n^{\frac{\ell-1}{k-1}}$ and
$\mathcal{H}$ be a $(n,k,\ell,\lambda)$-system, where $0< \lambda < n^{\frac{\ell-1}{k-2}-\delta}$.

Let $j\in[\ell-1]$ and $S \subset V(\mathcal{H})$ be a set of size $j$.
Since $\mathcal{H}$ is an $(n,k,\ell,\lambda)$-system,
$L_{\mathcal{H}}(S)$ is an $(n,k-j,\ell-j,\lambda)$-system.
Therefore,
\begin{align}
\Delta(\mathcal{H}) & \le \lambda{\binom{n}{\ell-1}}/{\binom{k-1}{\ell-1}}< t^{k-1}, \quad{\rm and}\quad \notag\\
|\Delta_{j}(\mathcal{H})| & \le \lambda {\binom{n}{\ell-j}}/{\binom{k-j}{\ell-j}} = O(\lambda n^{\ell-j})
\quad{\rm for}\quad 2 \le j \le \ell-1. \notag
\end{align}
It follows that
\begin{align}
C_{\mathcal{H}}(2,j)
= O\left(\lambda n^{\ell-j} |\mathcal{H}|\right)
= O\left(\lambda^2 n^{2\ell-j}\right)
\le nt^{2k-j-1-\epsilon} \quad{\rm for}\quad 2 \le j \le \ell-1. \notag
\end{align}
On the other hand, for $\ell \le j' \le k-1$ and a set $S \subset V(\mathcal{H})$ of size $j'$
the link $L_{\mathcal{H}}(S)$ has size at most $\lambda$.
Therefore,
\begin{align}
C_{\mathcal{H}}(2,j')
= O\left(\lambda |\mathcal{H}|\right)
= O\left(\lambda^2 n^{\ell}\right)
\le nt^{2k-j'-1-\epsilon}
\quad{\rm for}\quad \ell \le j' \le k-1. \notag
\end{align}
Therefore, by Theorem~\ref{THM:DLR-Uncrowd},
$\alpha(\mathcal{H})
= \Omega\left(\left(\log t\right)^{1/(k-1)}n/t\right)
= \Omega\left(\lambda^{-\frac{1}{k-1}} n^{\frac{k-\ell}{k-1}}\left(\log n\right)^{\frac{1}{k-1}}\right)$.
\end{proof}

\section{The Ramsey number of the $k$-Fan}\label{SEC:linear-tree}
In this section we prove Theorem~\ref{THM:Ramsey-k-fan}.
The lower bound (construction) is given by the so called {\em $L$-constructions}. These were introduced in~\cite{CM}, where they were used to answer an old Ramsey-type question of Ajtai-Erd\H os-Koml\'os-Szemer\'edi~\cite{AEKS}.

Let $m,n \ge 2$ and let $\mathcal{L}_{m,n}$ be the $k$-graph with vertex set $[m] \times [n]$ and edge set
\begin{align}
\left\{\{(x_1,y_1),(x_1,y_2),\ldots,(x_{k-1},y_2)\}: x_1< \cdots < x_{k-1}, y_1> y_2 \right\}. \notag
\end{align}

\begin{proposition}\label{PROP:L-construction-Fk-free}
For every $m,n \ge 2$ the hypergraph $\mathcal{L}_{m,n}$ is $F^{k}$-free.
\end{proposition}
\begin{proof}[Proof of Proposition~\ref{PROP:L-construction-Fk-free}]
Suppose that $\mathcal{L}_{m,n}$ contains a copy of $F^{k} = \{E_1,\ldots,E_k,E\}$.
Let $v = \bigcap_{i=1}^{k}E_i$ and assume that $v = (x_0,y_0)$,
$E = \left\{(x_1,y_1),(x_1,y_2),\ldots,(x_{k-1},y_2)\right\}$,
where $x_1< \cdots < x_{k-1}$ and $ y_1> y_2$.

By the definition of $F^{k}$, for every vertex $u \in E$, there exists an edge $E_i$ that contains both $u$ and $v$.
It is easy to see that if $x'_1 < x'_2$ and $y'_1 < y'_2$, then there is no edge in $\mathcal{L}_{m,n}$
containing both $(x'_1,y'_1)$ and $(x'_2,y'_2)$.
Therefore, we must have (see Figure~\ref{Fig:L-construction})
\begin{itemize}
\item[(1)] $x_0 \le x_1$ and $y_0 \ge y_1$, or
\item[(2)] $x_0 \ge x_{k-1}$ and $y_0 \le y_2$, or
\item[(3)] $x_0 = x_1$ and $y_2<y_0<y_1$, or
\item[(4)] $y_0 = y_2$ and $x_1<x_0<x_{k-1}$.
\end{itemize}

If $x_0 \le x_1$ and $y_0 \ge y_1$, then by the definition of $\mathcal{L}_{m,n}$, there is a $(k-1)$-set $J \subset [k]$
such that $\bigcap_{j\in J}E_j = (x_0,y_2)$, a contradiction.
If $x_0 \ge x_{k-1}$ and $y_0 \le y_2$, then by the definition of $\mathcal{L}_{m,n}$,
there exist $\{i,j\} \subset [k]$ such that $E_i \cap E_j = (x_1,y_0)$, a contradiction.
Similarly, if Case $(3)$ or Case $(4)$ happens, then there exist
$\{i,j\} \subset [k]$ such that $E_i \cap E_j = (x_1,y_2)$, a contradiction.
\end{proof}

\begin{figure}[htbp]
\centering
\begin{tikzpicture}[xscale=5,yscale=5]
\draw [->] (0,0)--(1.5+0.1,0);
\draw[color=uuuuuu] (1.5+0.08,0-0.08) node {$x$};

\draw [->] (0,0)--(0,1.1);
\draw[color=uuuuuu] (-0.08,1+0.08) node {$y$};

\draw (0,1)--(1.5,1);
\draw (1.5,0)--(1.5,1);

\draw[line width=1pt]
(0.5,0.6)--(0.5,0.4)--(1,0.4);

\draw [line width=0.8pt,color=sqsqsq,fill=sqsqsq,fill opacity=0.25]
(0,1)--(0.5,1)--(0.5,0.6)--(0,0.6)--(0,1);

\draw [line width=0.8pt,color=sqsqsq,fill=sqsqsq,fill opacity=0.25]
(1,0.4)--(1.5,0.4)--(1.5,0)--(1,0)--(1,0.4);

\begin{scriptsize}
\draw [fill=uuuuuu] (0.5,0.6) circle (0.3pt);
\draw [fill=uuuuuu] (0.5,0.4) circle (0.3pt);
\draw [fill=uuuuuu] (0.6,0.4) circle (0.3pt);
\draw [fill=uuuuuu] (0.7,0.4) circle (0.3pt);
\draw [fill=uuuuuu] (0.8,0.4) circle (0.3pt);
\draw [fill=uuuuuu] (1,0.4) circle (0.3pt);
\draw[color=uuuuuu] (0.75,0.4-0.05) node {$E$};
\draw[color=uuuuuu] (0.5+0.12,0.6) node {$(x_1,y_1)$};
\draw[color=uuuuuu] (0.5-0.1,0.4-0.05) node {$(x_1,y_2)$};
\draw[color=uuuuuu] (1,0.4+0.05) node {$(x_{k-1},y_{2})$};
\draw[color=uuuuuu] (0.75,0-0.07) node {$[m]$};
\draw[color=uuuuuu] (1.5+0.07,0.5) node {$[n]$};
\end{scriptsize}
\end{tikzpicture}
\caption{Only vertices that lie in these two shaded areas and the $L$-shaped path
that connects these two areas can be adjacent to all vertices in $E$.}
\label{Fig:L-construction}
\end{figure}
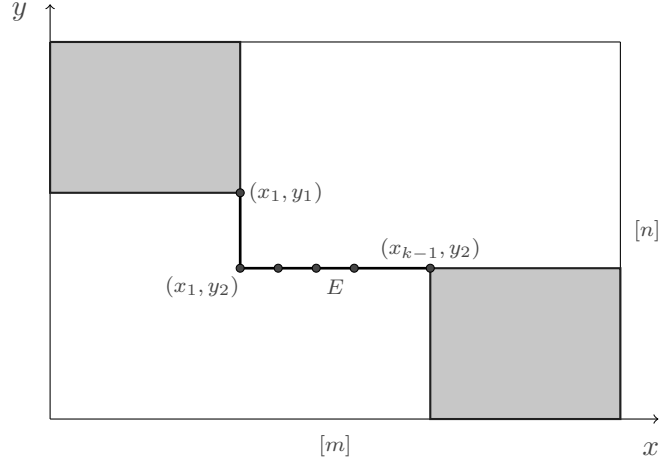

The following result gives an upper bound for the independence number of $\mathcal{L}_{m,n}$.

\begin{proposition}\label{PROP:L-construction-independent-set-2n}
The hypergraph $\mathcal{L}_{m,n}$ satisfies $\alpha(\mathcal{L}_{m,n}) < m+(k-2)n$.
\end{proposition}
\begin{proof}[Proof of Proposition~\ref{PROP:L-construction-independent-set-2n}]
Let $I$ be an independent in $\mathcal{L}_{m,n}$.
Remove the topmost vertex of each column and the $k-2$ rightmost vertices of each row in $I$.
It is easy to see that we removed at most $m + (k-2)n$ vertices from $I$,
and $I$ has no vertex left since otherwise $I$ would contain an edge in $\mathcal{L}_{m,n}$.
Therefore, $\alpha(\mathcal{L}_{m,n}) < m+(k-2)n$.
\end{proof}

Now we finish the proof of Theorem~\ref{THM:Ramsey-k-fan}.

\begin{proof}[Proof of Theorem~\ref{THM:Ramsey-k-fan}]
First we prove the lower bound.
Let $m = \lfloor\frac{t}{2}\rfloor$ and $n = \lfloor \frac{t-1}{2(k-2)} \rfloor$.
By Propositions~\ref{PROP:L-construction-Fk-free} and \ref{PROP:L-construction-independent-set-2n},
the $k$-graph $\mathcal{L}_{m,n}$ is $F^k$-free
and $\alpha(\mathcal{L}_{m,n}) \le m + (k-2)n < t$.
So,
\begin{align}
r_{k}(F^{k},t) > mn= \left\lfloor \frac{t}{2} \right\rfloor \left\lfloor \frac{t-1}{2(k-2)} \right\rfloor. \notag
\end{align}

To prove the upper bound, let us show that $r_{k}(F^{k},t) \le r_{k}(S_t^{k},t)$ first.
Indeed, let $\mathcal{H}$ be a $k$-graph on $r_{k}(S_t^k,t)$ vertices.
We may assume that $\mathcal{H}$ does not contain an independent set of size $t$.
Then, there exist $t$ distinct edges $E_1,\ldots,E_{t}$ and a vertex $v$ in $\mathcal{H}$
such that $E_i \cap E_j = \{v\}$ for $1 \le i < j \le t$.
Let $S$ be a set that contains exactly one vertex from each $E_i \setminus \{v\}$ for $i\in[t]$.
Then $S$ has size $t$ and hence contains an edge in $\mathcal{H}$, and it implies that $\mathcal{H}$ contains a copy of $F^{k}$.
Therefore, $r_{k}(F^{k},t) \le r_{k}(S_t^k,t)$,
and it follows from Theorem~\ref{THM:Ramsey-linear-tree} that $r_{k}(F^{k},t) \le t(t-1)+1$.
\end{proof}
\bibliographystyle{abbrv}
\bibliography{ramseyfan}
\end{document}